\documentclass[a4paper,reqno]{amsart}

\usepackage{enumerate}
\usepackage[usenames]{color}

\usepackage[colorlinks=true]{hyperref}
\hypersetup{urlcolor=blue, linkcolor=blue, citecolor=red}

\numberwithin{equation}{section}

\newtheorem{thm}{Theorem}[section]
\newtheorem{lem}[thm]{Lemma}
\newtheorem{prop}[thm]{Proposition}

\theoremstyle{definition}
\newtheorem{rem}[thm]{Remark}
\newtheorem{opb}[thm]{Open problem}

\newcommand\R{{\mathbb R}}
\newcommand\C{{\mathbb C}}

\newcommand\Tma{T_{\mathrm{max}}}
\newcommand\Sma{S_{\mathrm{max}}}
\newcommand\Supp{{\mathrm{supp}}\, }
\newcommand\Cz{{C_0(\R^N )}}
\newcommand\Comp{{\mathrm{c}}}
\newcommand\DI{u_0 }
\newcommand\DIb{v_0 }

\newcommand\GVar {{\Psi }}
\newcommand\Tsem{{\boldsymbol{\mathcal T}}}
\newcommand\goto{\mathop{\longrightarrow}}

\newcommand\Eqdef{\stackrel{\text{\tiny def}}{=}}

\newcommand\Granda{A}
\newcommand\Const{\chi }
\newcommand\Avoir{\mu }
\newcommand\Neh{I}
\newcommand\Ene{E}
\newcommand\Var{V}
\newcommand\Mom{J}
\newcommand\Gst{Q}
\newcommand\Bca{W}
\newcommand\Hcal{{\mathcal H}}
\newcommand\Kcal{{\mathcal K}}

\newcommand\MScN[1]{\href{http://www.ams.org/mathscinet-getitem?mr=#1}{\nolinkurl{(#1)}}}
\newcommand\DOI[1]{\href{http://dx.doi.org/#1}{(doi: \nolinkurl{#1})}}
\newcommand\LINK[1]{\href{#1}{(link: \nolinkurl{#1})}}

\newenvironment{notation}{\medskip \noindent{\bf Notation. }}{}

\begin{document}

\title{Finite-time blowup for some nonlinear complex Ginzburg-Landau equations}

\def\shorttitle{complex Ginzburg-Landau equations}

\author[T. Cazenave]{Thierry Cazenave$^1$}
\address{$^1$Universit\'e Pierre et Marie Curie \& CNRS, Laboratoire Jacques-Louis Lions,
B.C. 187, 4 place Jussieu, 75252 Paris Cedex 05, France}
\email{\href{mailto:thierry.cazenave@upmc.fr}{thierry.cazenave@upmc.fr}}

\author[S. Snoussi]{Seifeddine Snoussi$^2$}
\address{$^2$Institut Pr\'eparatoire aux Etudes Scientifiques et Techniques, Universit\'e de Car\-tha\-ge, B.P. 51, 2070 La Marsa, Tunisia}
\email{\href{mailto:seifsnoussi@yahoo.fr}{seifsnoussi@yahoo.fr}}

\subjclass[2010] {Primary 35Q56; secondary 35B44, 35K91, 35Q55}

\keywords{Complex Ginzburg-Landau equation, finite-time blowup, energy, variance}

\begin{abstract}

In this article, we review  finite-time blowup criteria for the family of complex Ginzburg-Landau equations
$u_t =  e^{ i\theta } [\Delta u +  |u|^\alpha u] + \gamma u$ on ${\mathbb R}^N $, where $0 \le  \theta \le \frac {\pi } {2}$, $\alpha >0$ and $\gamma  \in {\mathbb R} $. 
We study in particular the effect of the parameters $\theta $ and $\gamma $, and the dependence of the blowup time on these parameters.

\end{abstract}

\maketitle

\setcounter{tocdepth}{3}

\tableofcontents

\section{Introduction}

In this paper, we review certain known results, and present some new ones, on the problem of finite-time blowup for the family of complex Ginzburg-Landau equations
on $\R^N $
\begin{equation} \label{GL}
\begin{cases}
 u_t =  e^{ i\theta } [\Delta u +  |u|^\alpha u] + \gamma u \\
u(0)= \DI
\end{cases}
\end{equation}
where  $ 0\le \theta \le \frac {\pi } {2}$\footnote{One might consider $-\frac {\pi } {2} \le \theta \le 0$,  which is equivalent, by changing $u$ to $ \overline{u} $}, $\alpha >0$ and $\gamma  \in \R $\footnote{For a general $\gamma \in \C$, the imaginary part is eliminated by changing $u(t,x)$ to $ e^{ - i t \Im \gamma } u(t,x)$}.
The case $\theta =0$ of equation~\eqref{GL} is the well known nonlinear heat equation
\begin{equation} \label{NLH}
\begin{cases}
 u_t =   \Delta u +  |u|^\alpha u + \gamma u \\
u(0)= \DI
\end{cases}
\end{equation}
which arises in particular in chemistry and biology. See e.g.~\cite{Fife}. 
The case $\theta =  \frac {\pi } {2}$ of~\eqref{GL} is the equally well known nonlinear Schr\"o\-din\-ger equation
\begin{equation} \label{NLS}
\begin{cases}
 u_t =   i [\Delta u +  |u|^\alpha u] + \gamma u \\
u(0)= \DI
\end{cases}
\end{equation}
which is an ubiquitous model for weakly nonlinear dispersive waves and nonlinear optics. See e.g.~\cite{SulemS}. 
Therefore, equation~\eqref{GL} can be considered as ``intermediate" between the nonlinear heat and Schr\"o\-din\-ger equations. 
Equation~\eqref{GL}  is itself a particular case of the more general complex Ginzburg-Landau equations
on $\R^N $
\begin{equation} \label{GLG}
\begin{cases}
 u_t =  e^{ i\theta } \Delta u + e^{ i\phi } |u|^\alpha u + \gamma u \\
u(0)= \DI
\end{cases}
\end{equation}
where $ 0\le \theta \le \frac {\pi } {2}$, $\phi \in \R$, $\alpha >0$ and $\gamma \in \R$, 
which is a generic modulation equation describing the nonlinear evolution of patterns at near-critical conditions. See e.g.~\cite{StewartsonS, CrossH, Mielke}.

Two  strategies have been developed for studying finite-time blowup.
The first one consists in deriving conditions on the initial value, as general as possible, which ensure that the corresponding solution of~\eqref{GL} blows up in finite time. 
The proofs  often use a differential inequality which is satisfied by some quantity related to the solution, and one shows that this differential inequality can only hold on a finite-time interval. 
The major difficulty is to guess the appropriate quantity to calculate. However, when such a method can be applied, it usually provides a simple proof of blowup, under explicit conditions on the initial value. 
On the other hand, this strategy does not give any information on how the blowup occurs, nor on the mechanism that leads to blowup.
Concerning the family~\eqref{GL}, this is the type of approach used in~\cite{Kaplan, Levine, Ball} for the nonlinear heat equation; in~\cite{Zakharov, Glassey, Kavian, Tsutsumi, OgawaTu, OgawaTd} for the nonlinear Schr\"o\-din\-ger equation; in~\cite{SnoussiT, CazenaveDW, CazenaveDF} for the intermediate case of~\eqref{GL}.

Another strategy consists in looking for an ansatz of an approximate blowing-up solution, and then showing that the remainder remains bounded, or becomes small with respect to the approximate solution, as time tends to the blow-up time of the approximate solution. The first difficulty is to find the appropriate ansatz. Then, proving the boundedness of the remainder is often quite involved technically.
When this method is successful, it provides a precise description of how the corresponding solutions blow up. It may also explain the mechanism that makes these solution blow up. For the family~\eqref{GL}, this is the strategy employed in particular in~\cite{MerleZ} for the nonlinear heat equation; in~\cite{MerleR2, MerleR3, MerleR4, Raphael, RaphaelZ, MerleRZ} for the nonlinear Schr\"o\-din\-ger equation; in~\cite{Zaag, MasmoudiZ} for the intermediate case of~\eqref{GL} (and even~\eqref{GLG}). 

Equation~\eqref{GL} enjoys certain properties which the general equation~\eqref{GLG} does not. in particular its solutions satisfy certain energy identities (see Section~\ref{sCauchy}), which make it possible to study blowup by the first approach described above. We review sufficient conditions for finite-time blowup
(obtained using this approach),  and we study the influence of the  parameters $\theta $ and $\gamma $. 
In Section~\ref{sHeat} and~\ref{sNLS}, we recall the standard results for the heat equation~\eqref{NLH}  and the Schr\"o\-din\-ger equation~\eqref{NLS}, respectively. 
We are not aware of any previous reference for Theorem~\ref{eNLH4}, nor for the case $\gamma >0$ of Theorem~\ref{eNLS2}, although the proofs use standard arguments. 
The case $\gamma >0$ of Theorem~\ref{eNLS3} seems to be new.
Section~\ref{sGL} is devoted to the complex Ginzburg-Landau equation~\eqref{GL}. 
In Section~\ref{sBlowup}, we review sufficient conditions for blowup, and the case $\gamma <0$ of Theorem~\ref{eBUg1} is partially new.
Finally, we study in Section~\ref{sBUT} the behavior of the blowup time as the parameter $\theta $ approaches $ \frac {\pi } {2}$, i.e. as the equation gets close to the nonlinear Schr\"o\-din\-ger equation~\eqref{NLS}. We consider separately the cases $\alpha <\frac {4} {N}$ (subsection~\ref{sGP}) and $\alpha \ge \frac {4} {N}$ (subsection~\ref{sGN}). The case $\gamma >0$ of Theorem~\ref{eGGL1bu}, and Theorem~\ref{eGGLhbu}, are new. A few open question are collected in Section~\ref{OPPB}.

\begin{notation} 
We denote by $L^p (\R^N ) $, for $1\le p\le \infty $, the usual (complex valued) Lebesgue spaces. $H^1 (\R^N  ) $ and $H^{-1} (\R^N ) $ are the usual (complex valued) Sobolev spaces. 
(See e.g.~\cite{AdamsF} for the definitions and properties of these spaces.)
We denote by $C^\infty _\Comp (\R^N )$ the set of (complex valued) functions that have compact support and are of class $C^\infty $. We denote by $\Cz$ the closure of $C^\infty _\Comp (\R^N  )$ in $L^\infty  (\R^N  ) $. 
In particular, $\Cz$ is the space of functions $u$ that are continuous $\R^N  \to \C$ and such that  $u(x)\to 0$ as  $ |x|\to \infty $.  $\Cz$ is endowed with the sup norm.
\end{notation}

\section{The Cauchy problem and energy identities} \label{sCauchy} 
For the study of the local well posedness of~\eqref{GL}, it is convenient to consider the equivalent integral formulation, given by Duhamel's formula,
\begin{equation} \label{GLI} 
u (t) = \Tsem_\theta (t) \DI +  \int _0^t \Tsem_\theta  (t-s) [e^{i\theta } |u (s)|^\alpha u(s) + \gamma u(s)] \, ds
\end{equation} 
where $(\Tsem _\theta (t)) _{ t\ge 0 }$ is the semigroup of contractions on $L^2 (\R^N ) $ generated by the operator $e^{i\theta } \Delta $ with domain $H^2 (\R^N ) $. Moreover, 
$\Tsem_\theta (t)\psi  = G_\theta (t) \star \psi $, where the kernel $G_\theta (t)$ is defined by
\begin{equation*} 
G_\theta (t)  (x)\equiv  (4\pi t e^{i\theta })^{-\frac {N} {2}} e^{- \frac { |x|^2} {4t e^{i\theta }}}.
\end{equation*} 
If $ 0\le \theta < \frac {\pi } {2}$, it is not difficult to show that $(\Tsem _\theta (t)) _{ t\ge 0 }$ is an analytic semigroup on $L^2 (\R^N ) $, and a bounded $C_0$ semigroup on $L^p (\R^N ) $ for $1\le p<\infty $, and on $\Cz$. 
In particular (see~\cite[formula~(2.2)]{CazenaveDW})
\begin{equation} \label{fUB1} 
 \| \Tsem _\theta (t) \DI \| _{ L^\infty  } \le (\cos \theta )^{-\frac {N} {2}}  \| \DI \| _{ L^\infty  }.
\end{equation} 
It is immediate by a contraction mapping argument (see~\cite[Theorem~1]{Segal}) that the Cauchy problem \eqref{GL} is locally well  posed in $\Cz$. Moreover, it is easy to see using the analyticity of the semigroup that $\Cz \cap H^1 (\R^N ) $ is preserved under the action of~\eqref{GL}. The following result is established in~\cite[Proposition~2.1 and Remark~2.2]{CazenaveDW} in the case $\gamma =0$, and the argument equally applies when $\gamma \not = 0$.

\begin{prop} \label{eLocGL1} 
Suppose $ 0\le \theta < \frac {\pi } {2}$. 
Given any $\DI \in \Cz $, there exist $T>0$ and a unique solution $u \in  C([0,T], \Cz )$ of~\eqref{GLI} on $ (0,T)$. 
Moreover, $u$ can be extended to a maximal interval $[0, \Tma)$, and if $\Tma <\infty $, then $  \|u(t)\| _{ L^\infty  }\to \infty $ as $t \uparrow \Tma$. If, in addition,  $\DI \in H^1 (\R^N ) $, then $u \in  C([0,T],  H^1 (\R^N ) ) 
\cap C((0,T), H^2 (\R^N ) )\cap C^1((0, T), L^2 (\R^N ) ) $ and  $u$ satisfies~\eqref{GL} in $L^2 (\R^N ) $ for all $t\in (0,T)$. Furthermore, if $\alpha < \frac {4} {N}$ and $\Tma <\infty $, then  $  \|u(t)\| _{ L^2  }\to \infty $ as $t \uparrow \Tma$.
\end{prop} 

\begin{rem} \label{eLocGL3} 
Whether the solution given by Proposition~\ref{eLocGL1} is global or not is discussed throughout this paper,
but we can  observe that, given $0\le \theta < \frac {\pi } {2}$ and  $\DI\in \Cz$, the corresponding solution of~\eqref{GL} is global if $\gamma $ is sufficiently negative. 
More precisely, if $\gamma < -\frac {1} {\alpha } [2 (\cos \theta )^{-\frac {N} {2}}  \| \DI \| _{ L^\infty  }]^{\alpha +1}$, then the corresponding solution $u$ of~\eqref{GL} is global and satisfies $  \| u( t) \| _{ L^\infty  } \le   2 (\cos \theta )^{-\frac {N} {2}} e^{\gamma  t} \| \DI \| _{ L^\infty  } $ for all $t\ge 0$.
Indeed,  $v(t) = e^{ - \gamma t} u(t)$ satisfies $v_t = e^{i\theta } ( \Delta v +  e^{\gamma \alpha t} | v |^\alpha v) $, so that
\begin{equation*} 
v(t) = \Tsem_\theta  (t) \DI + \int _0^t e^{\gamma \alpha s} \Tsem_\theta   (t-s) [  |v (s)|^\alpha v(s) ] \, ds .
\end{equation*} 
Setting $\phi (t)= \sup \{  \|v(s)\| _{ L^\infty  };\, 0\le s\le t \}$, it follows from~\eqref{fUB1}  that $\phi (t) \le c  \| \DI \| _{ L^\infty  } + \frac {c} {-\gamma \alpha } \phi (t)^{\alpha +1}$ with $c=  (\cos \theta )^{-\frac {N} {2}} $. 
Therefore,  if  $\gamma < - \frac {1 } {\alpha } (2c)^{\alpha +1}   \| \DI \| _{ L^\infty  }^\alpha $, then $\phi (t) \le 2c  \| \DI \| _{ L^\infty  }$ for all $0\le t<\Tma$, and the desired conclusion follows.
\end{rem} 

It $ \theta = \frac {\pi } {2}$, then~\eqref{GL} is the nonlinear Schr\"o\-din\-ger equation, and $(\Tsem _\theta (t)) _{ t\ge 0 }$ is  a group of isometries (which is not analytic). More restrictive conditions are needed for the local solvability of~\eqref{GL}, and the proofs make use of Strichartz's estimates. The following result is proved in~\cite[Theorem~I]{Kato1}. (Except for the blowup alternatives, which follow from~\cite[Theorems~4.4.1 and 4.6.1]{CLN10}.)

\begin{prop} \label{eLocGL2} 
Suppose $\theta =\frac {\pi } {2}$ and $(N-2)\alpha <4$.
Given any $\DI \in  H^1 (\R^N ) $, there exist $T>0$ and a unique  $u \in  C([0,T], H^1 (\R^N ) ) 
\cap  C^1((0, T), H^{-1} (\R^N ) ) $ which satisfies~\eqref{GL} for all $t\in [0,T]$
and such that $u(0)= \DI$. 
Moreover, $u$ can be extended to a maximal interval $[0, \Tma)$, and if $\Tma <\infty $, then $  \|u(t)\| _{ H^1  }\to \infty $ as $t \uparrow \Tma$. In addition, if $\alpha < \frac {4} {N}$ and $\Tma <\infty $, then  $  \|u(t)\| _{ L^2  }\to \infty $ as $t \uparrow \Tma$.
\end{prop} 

As observed above, an essential feature of equation~\eqref{GL} is the energy identities satisfied by its solutions. Set
\begin{gather} 
\Neh (w)=  \int _{\R^N }  |\nabla w|^2 - \int _{\R^N }  |w|^{\alpha +2}  \label{fDI1}  \\
\Ene (w)= \frac {1} {2}\int _{\R^N }  |\nabla w|^2 -\frac {1} {\alpha +2}\int _{\R^N }  |w|^{\alpha +2}  \label{fDE1} 
\end{gather} 
The functionals $\Neh$ and $\Ene$ are well defined on $ \Cz \cap H^1  (\R^N ) $; and if $(N-2) \alpha \le 4$, they are well defined on $H^1 (\R^N ) $.

Suppose $ 0\le \theta < \frac {\pi } {2}$, let $\DI \in \Cz \cap H^1 (\R^N ) $ and let $u$ be the corresponding solution of~\eqref{GL} defined on the maximal interval $[0, \Tma)$, given by Proposition~\ref{eLocGL1}. Multiplying the equation by $ \overline{u} $, we obtain
\begin{equation} \label{fGGLsbumu}
 \int _{\R^N }  \overline{u}  u_t =  \gamma \int  _{ \R^N  }  |u|^2 - e^{i\theta } \Neh (u)
\end{equation} 
and in particular, taking the real part,
\begin{equation} \label{fGGLs}
\frac {d} {dt} \int _{\R^N }  |u|^2= 2\gamma \int  _{ \R^N  }  |u|^2  - 2\cos \theta  \Neh (u)
\end{equation} 
for all $0< t< \Tma$. Multiplying the equation by $e^{-i\theta } u_t$, taking the real part and using~\eqref{fGGLsbumu} yields 
\begin{equation} \label{fGGLqbu}
\frac {d} {dt} \Ene (u(t)) = 
 - \cos \theta \int  _{ \R^N  } |u_t|^2 + \gamma ^2 \cos \theta \int  _{ \R^N  } |u|^2 - \gamma \cos (2 \theta ) \Neh (u).
\end{equation} 
for all $0< t< \Tma$. 
Applying~\eqref{fGGLs}, we see that this is equivalent to 
\begin{equation} \label{fGGLq}
\frac {d} {dt}  \Bigl[ \Ene (u(t)) - \frac {\gamma } {2} \cos \theta \int  _{ \R^N  } |u|^2 \Bigr] = 
 - \cos \theta \int  _{ \R^N  } |u_t|^2 + \gamma \sin ^2 \theta \Neh (u).
\end{equation} 

Suppose now $ \theta =  \frac {\pi } {2}$ and  $(N-2)\alpha <4$, let $\DI \in  H^1 (\R^N ) $ and let $u$ be the corresponding solution of~\eqref{GL} defined on the maximal interval $[0, \Tma)$, given by Proposition~\ref{eLocGL2}. 
Identities corresponding to~\eqref{fGGLsbumu}, \eqref{fGGLs} and~\eqref{fGGLqbu} hold. More precisely, the functions $t\mapsto  \| u(t) \| _{ L^2 }^2$ and $t\mapsto  \Ene( u(t))$ are $C^1$ on $[0, \Tma)$, and
\begin{gather} 
 \int _{\R^N }  \overline{u}  u_t =  \gamma \int  _{ \R^N  }  |u|^2 - i \Neh (u) \label{fSL0} \\
\frac {d} {dt} \int _{\R^N }  |u|^2= 2\gamma \int  _{ \R^N  }  |u|^2  \label{fSL1} \\
\frac {d} {dt} \Ene (u(t)) =   \gamma \Neh (u) \label{fSL2} 
\end{gather} 
for all $0\le t<\Tma$. 
Identity~\eqref{fSL0} is obtained by taking the $H^{-1}-H^1$ duality product of the equation with $  \overline{u} $ (the term $ \int _{\R^N }  \overline{u}  u_t$ is understood as the duality bracket $\langle u_t, u \rangle  _{ H^1, H^{-1} }$). \eqref{fSL1} follows, by taking the real part.
Identity~\eqref{fSL2} is formally obtained by multiplying the equation by $e^{-i\theta } u_t$ and taking the real part. However, the solution is not smooth enough to do so, thus a regularization process is necessary. See~\cite{Ozawa} for a simple justification. 
Still in the case of the Schr\"o\-din\-ger equation $ \theta =  \frac {\pi } {2}$, an essential tool in the blowup arguments is the variance identity. It concerns the variance
\begin{equation} \label{fVar1} 
\Var (w) = \int  _{ \R^N  }  |x|^2  |w|^2
\end{equation} 
which is not defined on $  L^2 (\R^N ) $, but on the weighted space $ L^2 (\R^N ,  |x|^2 dx) $.
It can be proved that if $\DI \in H^1 (\R^N ) \cap L^2 (\R^N ,  |x|^2 dx) $, then the corresponding solution $u$ of~\eqref{GL}  satisfies $u\in C([0, \Tma ), L^2 (\R^N ,  |x|^2 dx) )$. Moreover, the map $t \mapsto \Var (u(t))$ is $C^2$ on $[0, \Tma) $ and 
\begin{gather} 
\frac {d} {dt} \Var ( u(t) ) = -4\Mom (u(t))  + 2\gamma \Var (u(t)) \label{fSL3} \\
\frac {d} {dt} \Mom (u(t))= -2 \int  _{ \R^N  }  |\nabla u|^2 + \frac {N\alpha } {\alpha +2} \int  _{ \R^N  }  |u|^{\alpha +2} + 2\gamma \Mom (u (t) )  \label{fSL4}
\end{gather} 
for all $0\le t<\Tma$, where the functional $\Mom$ is defined by
\begin{equation}  \label{fSL5}
\Mom (w)= \Im \int  _{ \R^N  } ( x\cdot \nabla  \overline{w}  )w
\end{equation} 
for $w\in H^1 (\R^N )  \cap L^2 (\R^N ,  |x|^2 dx) $. The proof of these properties require appropriate regularizations and multiplications, see~\cite[Section~6.5]{CLN10}.

\section{The nonlinear heat equation} \label{sHeat} 

In this section, we consider the nonlinear heat equation~\eqref{NLH}. 
The first blowup result was obtained by Kaplan~\cite{Kaplan}. Its argument applies to positive solutions of the equation set on a bounded domain, and is based on a differential inequality satisfied by the scalar product of the solution with the first eigenfunction. It is easy to extend the argument to the equation set on $\R^N $.
Let $w(x) \equiv e^{-\sqrt{N^2+|x|^2}}$, so that $\Delta w \ge -w$ by  elementary calculations. 
If $\psi =  \| w \| _{ L^1 }^{-1} w$ and $\psi _\lambda (x) = \lambda ^N \psi (\lambda x)$ for $\lambda >0$, then
$ \|\psi _\lambda \| _{ L^1 }=1$ and $\Delta \psi _\lambda  \ge -\lambda ^2\psi_\lambda $. 
Let now $\DI \in \Cz \cap H^1 (\R^N ) $, $\DI \ge 0$, $\DI \not \equiv 0$, and let $u$ be the corresponding solution of~\eqref{NLH} defined on the maximal interval $[0, \Tma)$. The maximum principle implies that $u(t) \ge 0$ for all $0\le t<\Tma$. Multiplying the equation by $\psi _\lambda $, integrating by parts on $\R^N $ and using Jensen's inequality, we obtain
\begin{equation*} 
\begin{split} 
\frac {d} {dt}\int  _{ \R^N  } u\psi _\lambda  & = \int  _{ \R^N  }u \Delta \psi _\lambda + \int  _{ \R^N  }u ^{\alpha +1} \psi _\lambda + \gamma  \int  _{ \R^N  } u\psi _\lambda \\
& \ge (\gamma - \lambda ^2 )  \int  _{ \R^N  } u\psi _\lambda +  \Bigl(  \int  _{ \R^N  } u\psi _\lambda \Bigr)^{\alpha +1}.
\end{split} 
\end{equation*} 
It follows that $f(t) = \int  _{ \R^N  } u\psi _\lambda $ satisfies
\begin{equation} \label{fKap1} 
\frac {df} {dt} \ge (\gamma - \lambda ^2 + f^\alpha )  f 
\end{equation} 
on $[0,\Tma)$. It is not difficult to show that if $f(0) ^\alpha > \lambda ^2 -\gamma $, then~\eqref{fKap1} can only hold on a finite interval, so that $\Tma <\infty $. 
Therefore, we can distinguish two cases. If $\gamma \le 0$, we choose for instance $\lambda =1$ and we see that if $\DI$ is sufficiently ``large" so that $\int  _{ \R^N  } \DI \psi _1>  (1-\gamma )^{\frac {1} {\alpha }}$, then the solution blows up in finite time. 
If $\gamma >0$, then we let $\lambda = \sqrt \gamma $, so that the condition $f(0) ^\alpha > \lambda ^2 -\gamma $ is always satisfied if $\DI \not \equiv 0$. In this case, we see that every nonnegative, nonzero initial value produces a solution of~\eqref{NLH} which blows up in finite time.

Levine~\cite{Levine} established blowup by a different argument. It is based on a differential inequality satisfied by the $L^2$ norm of the solution, derived from the energy identities. This argument applies to sign-changing solutions and, more generally, to complex valued solutions, and to the equation set on any domain, bounded or not. 
Strangely enough, even though Kaplan's argument seems to indicate that blowup is more likely to happen if $\gamma >0$, it turns out that Levine's result only applies to the case $\gamma \le 0$, which we consider first.

\subsection{The case $\gamma \le 0$} \label{ssNLHGneg} 
It is convenient to set 
\begin{gather} 
\Neh _\gamma (w)=  \int _{\R^N }  |\nabla w|^2 - \int _{\R^N }  |w|^{\alpha +2}- \gamma  \int  _{ \R^N  } |u|^2  \label{fHeat3} \\
\Ene _\gamma (w)= \frac {1} {2}\int _{\R^N }  |\nabla w|^2 -\frac {1} {\alpha +2}\int _{\R^N }  |w|^{\alpha +2} -  \frac {\gamma } {2} \int  _{ \R^N  } |u|^2  \label{fHeat4} 
\end{gather} 
for $w\in \Cz \cap H^1 (\R^N ) $.

\begin{thm}[\cite{Levine}] \label{eNLH1} 
Let $\DI \in \Cz \cap H^1 (\R^N ) $  and let $u$ be the corresponding solution of~\eqref{NLH} defined on the maximal interval $[0, \Tma)$, given by Proposition~$\ref{eLocGL1}$. 
If $\gamma \le 0$ and $\Ene _\gamma (\DI ) <0$, where $\Ene _\gamma $ is defined by~\eqref{fHeat4},  then $u$ blows up in finite time, i.e. $\Tma <\infty $.
\end{thm}

\begin{proof} 
We obtain a differential inequality on the quantity
\begin{equation} \label{fHeatb1} 
M(t)= \frac {1} {2} \int _0^t  \| u(s) \| _{ L^2 }^2 ds.
\end{equation} 
Formulas~\eqref{fGGLsbumu} and~\eqref{fGGLq} (with $\theta =0$) yield
\begin{gather}
 \int _{\R^N }  \overline{u}  u_t =   -  \Neh _\gamma (u)  \label{fHeat5} \\
\frac {d} {dt}  \Ene _\gamma (u(t))  =   -  \int  _{ \R^N  } |u_t|^2  \le 0. \label{fHeat6}
\end{gather} 
Identity~\eqref{fHeat6} implies
\begin{equation} \label{fHeat7}
 \Ene _\gamma (u(t))  + \int _0^t  \| u_t \| _{ L^2 }^2  = \Ene _\gamma  (\DI).
\end{equation} 
Moreover, it follows from~\eqref{fHeat3} and~\eqref{fHeat4} that
\begin{equation} \label{fHeat9} 
\Neh _\gamma (u(t))  \le (\alpha +2) \Ene _\gamma (u(t)) -  \frac { (- \gamma ) \alpha} {2}  \| u \| _{ L^2 }^2
\end{equation} 
so that by~\eqref{fHeat7},
\begin{equation}  \label{fHeat10} 
\Neh _\gamma (u(t)) \le  (\alpha +2) \Ene _\gamma (\DI ) -  \frac { (- \gamma ) \alpha} {2}  \| u \| _{ L^2 }^2 - (\alpha +2)  \int _0^t  \| u_t \| _{ L^2 }^2<0.
\end{equation} 
We deduce from~\eqref{fHeatb1}, \eqref{fHeat5}  and~\eqref{fHeat10} that
\begin{equation}  \label{fHeatb2} 
M'' (t)  = \Re  \int _{\R^N }  \overline{u}  u_t =   -  \Neh _\gamma (u) \ge -  (\alpha +2) \Ene _\gamma (\DI ) +  (\alpha +2)  \int _0^t  \| u_t \| _{ L^2 }^2 > 0.
\end{equation} 
All the above formulas hold for $0\le t<\Tma$. Assume now by contradiction that $\Tma =\infty $. 
We deduce in particular from~\eqref{fHeatb2} that
\begin{equation}  \label{fHeatb3} 
M' (t) \goto  _{ t\to \infty  } \infty , \quad M (t) \goto  _{ t\to \infty  } \infty .
\end{equation} 
It follows from~\eqref{fHeatb1}, \eqref{fHeatb2}, and Cauchy-Schwarz's inequality (in time and space) that
\begin{equation} \label{fHeatb4} 
\begin{split} 
M(t) M''(t) & \ge \frac {\alpha +2} {2} \Bigl(   \int _0^t  \| u \| _{ L^2 }^2  \Bigr)  \Bigl(  \int _0^t  \| u_t \| _{ L^2 }^2  \Bigr) \\
& \ge \frac {\alpha +2} {2} \Bigl( \int _0^t   \Bigl|  \int _{\R^N }  \overline{u}  u_t  \Bigr| \Bigr)^2   \ge \frac {\alpha +2} {2} \Bigl( \int _0^t   \Bigl|  \Re \int _{\R^N }  \overline{u}  u_t  \Bigr| \Bigr)^2 \\
& = \frac {\alpha +2} {2} \Bigl( \int _0^t  M'' (s) \Bigr)^2  = \frac {\alpha +2} {2} ( M'(t) - M' (0 ))^2 .
\end{split} 
\end{equation} 
We deduce from~\eqref{fHeatb3} that $\frac {\alpha +2} {2} ( M'(t) - M' (0 ))^2 \ge \frac {\alpha +4} {4}  M'(t) ^2$ for $t$ sufficiently large. Therefore~\eqref{fHeatb4} yields
\begin{equation*} 
M(t) M''(t) \ge  \frac {\alpha +4} {4}  M'(t) ^2
\end{equation*} 
which means that $M (t) ^{-\frac {\alpha } {4}}$ is concave for $t$ large. Since $M (t) ^{-\frac {\alpha } {4}} \to 0$ as $t\to \infty $ by~\eqref{fHeatb3}, we obtain a contradiction.
\end{proof} 

The proof of Theorem~\ref{eNLH1} does not immediately provide an estimate of $\Tma$ in terms of $\DI$. It turns out that a variant of that proof, given in~\cite[Proposition~5.1]{HarauxW} yields such an estimate.

\begin{thm} \label{eNLH2} 
Under the assumptions of Theorem~$\ref{eNLH1}$, we have
\begin{equation} \label{fHeatb5} 
\Tma \le 
 \begin{cases} 
 \displaystyle 
  \frac {    \| \DI \| _{ L^2 }^2 } { \alpha (\alpha +2) (- \Ene _\gamma  ( \DI ) ) }  & \gamma =0 \\
 \displaystyle 
 \frac {1} {-\gamma \alpha } \log \Bigl(  1+  \frac {   -2 \gamma   \| \DI \| _{ L^2 }^2 } { 2(\alpha +2) (- \Ene _\gamma  ( \DI ) )   - \gamma \alpha      \| \DI  \| _{ L^2 }^2 } \Bigr) & \gamma <0 .
 \end{cases} 
\end{equation} 
\end{thm} 

\begin{proof} 
Set 
\begin{equation}  \label{fHeat11}
f(t)=  \|u(t) \| _{ L^2 }^2, \quad e(t) = \Ene _\gamma (u(t))  .
\end{equation} 
We first obtain an upper bound on $e$ in terms of $f$, then a differential inequality on $f$. 
It follows from~\eqref{fHeat5} and~\eqref{fHeat9} that
\begin{equation} \label{fHeat11b1}
\frac {df} {dt} \ge  2(\alpha +2)   (-e) + (-\gamma )\alpha f  .
\end{equation}
Since $\frac {df} {dt} >0$ by~\eqref{fHeatb2}, we deduce from~\eqref{fHeat6}, Cauchy-Schwarz's inequality, \eqref{fHeat5} and~\eqref{fHeat11b1} that
\begin{equation}  \label{fHeat12}
\begin{split} 
- f \frac {de} {dt} & = \int  |u|^2 \int  |u_t|^2 \ge  \Bigl| \int  \overline{u}u_t  \Bigr|^2 
 = \frac {1} {4}    \Bigl( \frac {df} {dt} \Bigr)^2 \\
& \ge \frac {1} {2} \Bigl( - (\alpha +2) e +  \frac { (- \gamma ) \alpha} {2}  f  \Bigr)  \frac {df} {dt} .
\end{split} 
\end{equation} 
This means that
\begin{equation}  \label{fHeat13}
\frac {d} {dt} \Bigl(  -e f^{-\frac {\alpha +2} {2}} + \frac {-\gamma } {2} f^{-\frac {\alpha } {2}} \Bigr)\ge 0
\end{equation} 
so that
\begin{equation} \label{fHeat14}
-e + \frac {-\gamma  } {2 } f  \ge \eta f^{\frac {\alpha +2} {2}} 
\end{equation}
with
\begin{equation} \label{fHeat15}
\eta =  -\Ene_\gamma   ( \DI )  \| \DI \| _{ L^2 }^{- (\alpha +2)} + \frac {-\gamma  } {2 }  \| \DI \| _{ L^2 }^{-\alpha } >0.
\end{equation}
It follows from~\eqref{fHeat11b1} and~\eqref{fHeat14} that
\begin{equation*}
\frac {df} {dt} \ge  -2(-\gamma )  f +  2(\alpha +2) \eta   f^{\frac {\alpha +2} {2}}.
\end{equation*}
Therefore,
\begin{equation} \label{fHeat15b2}
\frac {d} {dt}  \Bigl[ (e^{ -2 \gamma t  } f)^{-\frac {\alpha } {2}} \Bigr] + \alpha (\alpha +2) \eta  e^{ \gamma \alpha t  }  \le 0.
\end{equation}
After integration, then letting $t\uparrow \Tma$, we deduce that
\begin{equation} \label{fHeat15b3}
\alpha (\alpha +2) \eta  \int _0^{\Tma } e^{ \gamma \alpha t  }dt  \le   f(0)^{-\frac {\alpha } {2}}.
\end{equation}
Expressing $\eta $ and $f(0)$ in terms of $\DI$, estimate~\eqref{fHeatb5} easily follows in both the cases $\gamma =0$ and $\gamma <0$.
\end{proof} 

\begin{rem} \label{eNLH3} 
Here are some comments on Theorems~\ref{eNLH1} and~\ref{eNLH2}.
\begin{enumerate}[{\rm (i)}] 

\item \label{eNLH3:0}  
Suppose $\DI \not = 0$ and $\Ene _\gamma ( \DI) = 0$. In particular,  $\Neh _\gamma ( \DI) <0$. Therefore, $\DI $ is not a stationary solution of~\eqref{NLH}, and it follows from~\eqref{fHeat6} that $\Ene _\gamma ( u(t) ) < 0$ for all $0<t< \Tma$. Thus  we can apply Theorems~\ref{eNLH1} and~\ref{eNLH2} with $\DI  $ replaced by $u( \varepsilon )$, and let $\varepsilon \downarrow 0$. In particular, we see that $\Tma <\infty $. Moreover, estimate~\eqref{fHeatb5} holds if $\gamma <0$. 

\item \label{eNLH3:1} Given any nonzero $\varphi \in \Cz \cap H^1 (\R^N ) $, we have $\Ene _\gamma (\lambda \varphi ) )<0$ if $ |\lambda |$ is sufficiently large. Thus we see that the sufficient condition $\Ene _\gamma ( \DI) < 0$ can indeed be achieved by certain initial values, for any  $\alpha >0$ and $\gamma \le 0$. 

\item \label{eNLH3:2} Let $\alpha >0$ and fix $\DI  \in \Cz \cap H^1 (\R^N )$. It is clear that if $\gamma $ is sufficiently negative, then $\Ene _\gamma ( \DI) \ge 0$, so that one cannot apply Theorem~\ref{eNLH1}. This is not surprising, 
since the corresponding solution of~\eqref{NLH} is global if $\gamma $ is sufficiently negative. (See Remark~\ref{eLocGL3}.)

\item \label{eNLH3:3} Suppose $\gamma <0$.
It follows from~\eqref{fHeatb5} and the assumption  $\Ene _\gamma  ( \DI )<0$ that
$\Tma \le  \frac {1} {-\gamma \alpha } \log (  1+  \frac {   2   } {   \alpha    } )$.
In particular, we see that for $\DI$ as in Theorem~\ref{eNLH2}, the blowup time is bounded in terms of $\alpha $ and $\gamma $ only, independently of $\DI$.

\end{enumerate} 
\end{rem} 

As observed in Remark~\ref{eNLH3}~\eqref{eNLH3:3}, in the case $\gamma <0$, Theorem~\ref{eNLH2} does not apply to solutions for which the blow-up time would be arbitrarily large. When
\begin{equation} \label{fPSb1}
(N-2) \alpha < 4
\end{equation}
this can be improved by using the potential well argument of Payne and Sattinger~\cite{PayneS}. 
To this end, we introduce some notation.
Assuming~\eqref{fPSb1} and $\gamma <0$, we denote by $\Gst _\gamma $ the unique positive, radially symmetric, $H^1$ solution of the equation
\begin{equation} \label{fPSb2}
-\Delta \Gst  - \gamma \Gst =  |\Gst |^\alpha \Gst ,
\end{equation}
and we recall below the following well-known properties of $\Gst _\gamma $. 

\begin{prop} \label{ePSb1}
Assume~\eqref{fPSb1} and $\gamma <0$, and let $\Gst _\gamma \in H^1 (\R^N ) $ be the unique positive, radially symmetric solution of~\eqref{fPSb2}.
\begin{enumerate}[{\rm (i)}]

\item \label{ePSb1:1} $\Ene_\gamma ( \Gst _\gamma  )>0$ and $\Neh_\gamma (\Gst _\gamma )=0$.

\item \label{ePSb1:2} $\displaystyle \Ene_\gamma (\Gst _\gamma )= \inf  \Bigl\{ \Ene_\gamma (v);\, v\in H^1 (\R^N ) , v\not = 0, \Neh _\gamma (v) =0 \Bigr\}$.

\item \label{ePSb1:3} If $u\in H^1 (\R^N ) $, $\Ene _\gamma (u) <  \Ene_\gamma (\Gst _\gamma )$ and $\Neh_\gamma (u) <0$, then $\Neh _\gamma ( u) \le - (\Ene_\gamma (\Gst _\gamma ) - \Ene _\gamma (u))$.
\end{enumerate}
\end{prop}

\begin{proof} 
The first two properties are classical, see for instance~\cite[Chapter~3]{Willem}.
Next, let $u\in H^1 (\R^N ) $ with $\Neh_\gamma (u)<0$, and set $h(t)= \Ene_\gamma  (tu)$ for $t>0$. 
It follows easily that $h'(t)=  \frac {1} {t} \Neh_\gamma (tu)$, so that there exists a unique $t^* >0$ such that $h$ is increasing on $[0, t^*]$ and decreasing and concave on $[t^*, \infty )$. In particular, $\Neh _\gamma (t^* u)=0$, thus $h (t^*)= \Ene _\gamma (t^*u) \ge \Ene_\gamma (\Gst _\gamma )$. Moreover, since $\Neh _\gamma (u)<0$ we have $t^*<1$ and by the concavity of $h$ on $[t^*, 1]$, 
\begin{equation*} 
\begin{split} 
\Ene _\gamma (u)= h(1) &\ge h (t^*) + (1-t^*) h '(1) = h (t^*) + (1-t^*) \Neh _\gamma (u) \\ & \ge \Ene_\gamma (\Gst _\gamma ) + (1-t^*) \Neh _\gamma (u) \ge \Ene_\gamma (\Gst _\gamma ) + \Neh _\gamma (u) 
\end{split} 
\end{equation*} 
from which~\eqref{ePSb1:3} follows.
\end{proof} 

We have the following result.

\begin{thm} \label{ePSb3}
Assume~\eqref{fPSb1}, $\gamma <0$, and let $\Gst _\gamma \in H^1 (\R^N ) $ be the unique positive, radially symmetric solution of~\eqref{fPSb2}.
Let $\DI \in \Cz \cap H^1 (\R^N ) $ and let $u$ be the corresponding solution of~\eqref{NLH} defined on the maximal interval $[0, \Tma)$.
If  $ \Ene _\gamma  (\DI ) < \Ene_\gamma (\Gst _\gamma )$ and $\Neh _\gamma ( \DI )<0$, then $u$ blows up in finite time, and
\begin{equation} \label{ePSb3:1}
\Tma \le \frac {1} {-\gamma \alpha  } \left[ \frac {(\alpha +4) [ \Ene _\gamma ( \DI ) ]^+} {\Ene_\gamma (\Gst _\gamma ) - \Ene _\gamma ( \DI )} +   \log  \Bigl(  \frac {2(\alpha +2)} {\alpha } \Bigr) \right].
\end{equation}
\end{thm}

\begin{proof}
The key observation is that 
\begin{equation} \label{fPSb3}
\Neh _\gamma ( u (t)) \le - (\Ene_\gamma (\Gst _\gamma ) - \Ene _\gamma ( \DI )) <0
\end{equation}
for $0\le t<\Tma$. Indeed, it follows from~\eqref{fHeat6} that $\Ene _\gamma  (u(t) ) \le  \Ene _\gamma ( \DI )< \Ene_\gamma (\Gst _\gamma )$. Therefore, Proposition~\ref{ePSb1}~\eqref{ePSb1:3} implies that~\eqref{fPSb3} holds as long as $\Neh _\gamma ( u (t)) <0$. Since the right-hand side of~\eqref{fPSb3} is a negative constant, we see by continuity and Proposition~\ref{ePSb1}~\eqref{ePSb1:2} that $\Neh _\gamma ( u (t)) $ must remain negative; and so~\eqref{fPSb3} holds for all $0\le t<\Tma$.

If $\Ene_\gamma  (\DI )\le  0$, then the result follows from Remark~\ref{eNLH3}~\eqref{eNLH3:1} and~\eqref{eNLH3:3}, so we suppose
\begin{equation}  \label{ePSb3:2}
0 <  \Ene _\gamma (\DI ) < \Ene _\gamma ( \Gst _\gamma  ).
\end{equation}
We use the notation~\eqref{fHeat11}  introduced in the proof of Theorem~\ref{eNLH2}.
It follows from~\eqref{fHeat5} and~\eqref{fPSb3} that $\frac {df} {dt} \ge 2 (\Ene_\gamma (\Gst _\gamma ) - \Ene _\gamma ( \DI )) $,  so that
\begin{equation} \label{fPSb4}
f (t) \ge 2 (\Ene_\gamma (\Gst _\gamma ) - \Ene _\gamma ( \DI )) t  .
\end{equation}
In particular, we see that
\begin{equation} \label{fPSb5}
\sigma  (t) \Eqdef - e (t)     + \frac {-\gamma  } {2 } f (t) \ge - \Ene _\gamma ( \DI ) -\gamma 
 (\Ene_\gamma (\Gst _\gamma ) - \Ene _\gamma ( \DI )) t   .
\end{equation}
We set 
\begin{equation}  \label{fPSb7}
\tau = \frac {(\alpha+4 )  \Ene _\gamma ( \DI )} {-\gamma \alpha ( \Ene_\gamma (\Gst _\gamma ) - \Ene _\gamma ( \DI ) )} >   \frac { \Ene _\gamma ( \DI )} {- \gamma ( \Ene_\gamma (\Gst _\gamma ) - \Ene _\gamma ( \DI ) )}.
\end{equation}
If $\Tma\leq \tau $ then~\eqref{ePSb3:1} follows from~\eqref{fPSb7}. We now suppose
\begin{equation}  \label{fPSb6}
\Tma >   \tau .
\end{equation}
Setting $\DIb = u(\tau )$, we see that the solution $v$ of~\eqref{NLH} with the initial condition $v(0) = \DIb$ is  $v(t)= u(t+\tau )$ and that its maximal existence time $\Sma$ is  $\Sma = \Tma -\tau $.
Since $\sigma  (\tau )>0$ by~\eqref{fPSb5} and~\eqref{fPSb7}  we can argue as in the proof of Theorem~\ref{eNLH2}
(note that $\eta>0$,  where $\eta $ is given by~\eqref{fHeat15} with $\DI$ replaced by $\DIb$), and we deduce (cf.~\eqref{fHeat15b3}) that
\begin{equation}  \label{fPSb8}
 \frac {2(\alpha +2) (- \Ene_\gamma   ( \DIb) )   - \gamma \alpha \| \DIb \| _{ L^2 }^2 } { 2(\alpha +2) (- \Ene _\gamma  ( \DIb ) )   - \gamma (\alpha +2)     \| \DIb  \| _{ L^2 }^2 } \le e^{ \gamma \alpha \Sma  } .
\end{equation}
(Observe that $\sigma (\tau )>0$, so that the denominator on the left-hand side of~\eqref{fPSb8} is positive.)
Note that by~\eqref{fPSb4}, \eqref{fPSb7}, and the fact that $e (t) $ is nonincreasing
\begin{equation*}
\| \DIb \| _{ L^2 }^2 =  \| u(\tau) \| _{ L^2 }^2 \ge 2  (\Ene_\gamma (\Gst _\gamma ) - \Ene _\gamma ( \DI )) \tau    \ge \frac { 2(\alpha +4) } {-\gamma \alpha } \Ene _\gamma ( \DI)  \ge \frac { 2(\alpha +4) } {-\gamma \alpha } \Ene _\gamma ( \DIb) 
\end{equation*}
from which it follows that
\begin{equation} \label{fPSb10}
 \frac {2(\alpha +2) (- \Ene_\gamma   ( \DIb) )   - \gamma \alpha \| \DIb \| _{ L^2 }^2 } { 2(\alpha +2) (- \Ene _\gamma  ( \DIb ) )   - \gamma (\alpha +2)     \| \DIb  \| _{ L^2 }^2 } \ge \frac {\alpha } {2(\alpha +2)} .
\end{equation}
\eqref{fPSb8} and~\eqref{fPSb10} yield $\Sma \le \frac {1} {\alpha  } \log  \Bigl( \frac {2(\alpha +2)} {\alpha } \Bigr)$.
Since $\Tma= \tau +\Sma$, the result follows by applying~\eqref{fPSb7}.
\end{proof}

\begin{rem}  \label{ePSb4}
Proposition~\ref{ePSb3} applies to solutions for which the maximal existence time is arbitrary large.
Indeed, given $\varepsilon >0$, let $\DI ^\varepsilon = (1+\varepsilon ) \Gst _\gamma $ and $u^\varepsilon $
the corresponding solution of~\eqref{NLH}. It is straightforward to verify that for all $\varepsilon >0$,
 $\Ene _\gamma (\DI ^\varepsilon)< \Ene _\gamma (\Gst _\gamma )$ and $\Neh _\gamma (\DI ^\varepsilon)<0$. Indeed, the function $\varepsilon\mapsto \Ene _\gamma ((1+\varepsilon ) \Gst _\gamma )$
 is decreasing on $[1,+\infty)$, $\Neh _\gamma (\DI ^\varepsilon)<(1+\varepsilon)^2 \Neh_\gamma (\Gst _\gamma )$ and $ \Neh_\gamma (\Gst _\gamma )=0$. 
  Hence $\DI^\varepsilon $ satisfies the assumptions of Proposition~\ref{ePSb3}.
On the other hand, $\Gst _\gamma $ is a stationary (hence global) solution of~\eqref{NLH}, so that the blowup
time of $v^\varepsilon $ goes to infinity as $\varepsilon \downarrow 0$, by continuous dependence.
\end{rem}

\subsection{The case $\gamma >0$} \label{ssNLHGpos} 
Levine's method  (Section~\ref{ssNLHGneg}) does not immediately apply  when  
$\gamma >0$, but it can easily be adapted, after a suitable change of variable.

\begin{thm} \label{eNLH4}
Suppose $\gamma > 0$. Let $\DI \in \Cz \cap H^1 (\R^N ) $ and let $u$ be the corresponding solution of~\eqref{NLH} defined on the maximal interval $[0, \Tma)$.
If $\Ene (\DI) <0$, where $\Ene $ is defined by~\eqref{fDE1},  then $u$ blows up in finite time, i.e., $\Tma <\infty $.
Moreover,
\begin{equation} \label{eNLH4:1}
\Tma \le \frac {1} {\alpha \gamma } \log  \Bigl( 1+  \frac { \gamma  \| \DI \| _{ L^2 }^2} { (\alpha +2)    (- \Ene (\DI ))  } \Bigr) <\infty .
\end{equation}
\end{thm}

\begin{proof}
We set $v(t)= e^{- \gamma t} u(t)$, so that
\begin{equation} \label{eNLH4:2}
\begin{cases}
 v_t=  \Delta v+ e^{\alpha \gamma t} |v|^\alpha v  \\
v(0)= \DI
\end{cases}
\end{equation}
and we use the arguments in the proof of Theorem~\ref{eNLH2}.
Setting
\begin{equation} \label{fNot1} 
 \widetilde{f}  =  \| v  \| _{ L^2 }^2 , \quad   
 \widetilde{\jmath} =   \| \nabla v \| _{ L^2 }^2 - e^{\alpha \gamma t}  \| v \| _{ L^{\alpha +2} }^{\alpha +2}  ,    \quad  
  \widetilde{e} = \frac {1} {2}  \| \nabla v \| _{ L^2 }^2 -\frac {e^{\alpha \gamma t}} {\alpha +2}  \| v \| _{ L^{\alpha +2} }^{\alpha +2} 
\end{equation}
it follows from~\eqref{eNLH4:2} that
\begin{equation} \label{eNLH4:3}
\int  _{ \R^N  }  \overline{v}v_t= -  \widetilde{\jmath} (t)
\end{equation}
and
$\frac {d  \widetilde{e} } {dt}   =  -    \int  _{ \R^N  } |v_t|^2  + \alpha \gamma  \widetilde{e}  -\frac {\alpha \gamma } {2} \int  _{ \R^N  } |\nabla u|^2$, so that
\begin{equation} \label{eNLH4:5}
\frac {d  \widetilde{e} } {dt}   -\alpha \gamma  \widetilde{e}  \le   -    \int  _{ \R^N  } |v_t|^2.
\end{equation}
In particular, 
\begin{equation}  \label{eNLH4:6}
 \widetilde{e} (t) \le e^{\alpha \gamma t}  \widetilde{e} (0)= e^{\alpha \gamma t} \Ene (\DI) <0 . 
\end{equation} Applying~\eqref{eNLH4:5}, Cauchy--Schwarz's inequality, and~\eqref{eNLH4:3}, we obtain
\begin{equation}  \label{eNLH4:7}
-  \widetilde{f}    \Bigl(    \frac {d  \widetilde{e} } {dt} -\alpha \gamma  \widetilde{e}   \Bigr) \ge    \int  |v|^2 \int  |v_t|^2 \ge    \Bigl| \int  \overline{v}v_t  \Bigr|^2   =      \widetilde{\jmath }  ^2  = \frac {1} {2}   (-  \widetilde{\jmath }   ) \frac {d \widetilde{f} } {dt} .
\end{equation}
Note that
\begin{equation}  \label{eNLH4:8}
 \widetilde{\jmath }  = (\alpha +2)  \widetilde{e}  -\frac {\alpha } {2} \int_{ \R^N  } |\nabla v|^2 \le  (\alpha +2)  \widetilde{e} .
\end{equation}
Since $\frac {d  \widetilde{f}} {dt} >0$ by~\eqref{eNLH4:3}, \eqref{eNLH4:8}  and~\eqref{eNLH4:6}, we deduce from~\eqref{eNLH4:7} and~\eqref{eNLH4:8} that
$-  \widetilde{f}   (  \frac {d \widetilde{e}} {dt} -\alpha \gamma  \widetilde{e}  )\ge -\frac {\alpha +2} {2}  \widetilde{e}     \frac {d  \widetilde{f}} {dt} $.
Therefore $\frac {d} {dt} [e^{- \alpha \gamma t} (-  \widetilde{e})  \widetilde{f}^{-\frac {\alpha +2} {2}}]\ge 0$, 
and so
\begin{equation*}
e^{- \alpha \gamma t} (- \widetilde{e} (t) ) \ge [- \widetilde{e} (0)] \widetilde{f}(0)^{-\frac {\alpha +2} {2}} \widetilde{f}(t)^{\frac {\alpha +2} {2}} = (- \Ene  (\DI) ) \|\DI\| _{ L^2 } ^{-(\alpha +2)}  \widetilde{f} (t)^{\frac {\alpha +2} {2}}.
\end{equation*}
Thus we see that
\begin{equation*}
\begin{split}
\frac {d \widetilde{f} } {dt} = - 2     \widetilde{\jmath }   \ge -2(\alpha +2)   \widetilde{e}   \ge 2(\alpha +2)   (- \Ene  (\DI) ) \|\DI\| _{ L^2 } ^{-( \alpha +2) } \widetilde{f}^{\frac {\alpha +2} {2}} e^{ \alpha \gamma t}.
\end{split}
\end{equation*}
This shows that $  \alpha  (\alpha +2)   (- \Ene (\DI) ) \|\DI\| _{ L^2 } ^{-( \alpha +2) }
 e^{  \alpha \gamma t }  +  \frac {d} {dt} ( \widetilde{f}^{- \frac {\alpha } {2}} ) \le 0$, and~\eqref{eNLH4:1} easily follows.
\end{proof}

\begin{rem}  \label{eNLH5} 
Below are a few comments on Theorem~\ref{eNLH4}.
\begin{enumerate}[{\rm (i)}] 

\item \label{eNLH5:1} Kaplan's calculations at the beginning of Section~\ref{sHeat} show that if $\gamma >0$, every nonnegative, nonzero initial value produces finite-time blowup. On the other hand, in dimension $N\ge 2$, there exist stationary solutions in $\Cz$, which are global solutions of~\eqref{NLH}. Indeed, it is not  difficult to prove that for every $\eta>0$, the solution $u$ of the ODE $u'' + \frac {N-1} {r} u' + \gamma u +  |u|^\alpha u=0$ with the initial conditions $u(0)= \eta$ and $u'(0) =0$ oscillates indefinitely and converges to $0$ as $r\to \infty $. This yields a solution $u\in C^2 (\R^N ) \cap \Cz$ of the equation $\Delta u+\gamma u+ |u|^\alpha u=0$, hence a stationary solution of~\eqref{NLH}. Whether or not there exist stationary solutions in $H^1 (\R^N ) \cap \Cz$ seems to be an open problem in general. Note also that in dimension $N=1$, there is no stationary solution in $\Cz$, this can be easily deduced from the resulting ODE. 

\item \label{eNLH5:2} In the case $\gamma =0$, $\alpha =\frac {2} {N}$ is the Fujita critical exponent. If $\alpha >\frac {2} {N}$, then small initial values in an appropriate sense give rise to global solutions of~\eqref{NLH}. On the other hand, if $\alpha \le \frac {2} {N}$, then every nonnegative, nonzero initial value produces finite-time blowup. (See~\cite{Fujita, Hayakawa,  KobayashiST, Weissler, Kavian}.) However, given any $\alpha \le \frac {2} {N}$, there exist nonzero initial values producing global solutions. 
In the one-dimensional case, they can be initial values that change sign sufficiently many times and are sufficiently small~\cite{MizoguchiY1, MizoguchiY2}. In any dimension, they can be self-similar solutions~\cite[Theorem~3]{HarauxW}.
If $\gamma >0$, then equation~\eqref{NLH} is not scaling-invariant, so that one cannot expect self-similar solutions.

\end{enumerate} 
\end{rem} 

\section{The nonlinear Schr\"o\-din\-ger  equation} \label{sNLS} 

In this section, we consider the nonlinear Schr\"o\-din\-ger equation~\eqref{NLS}. We assume $\alpha <\frac {4} {N-2}$, and it follows from Proposition~\ref{eLocGL2} that the Cauchy problem is locally well-posed in $H^1 (\R^N ) $. 
In contrast with the nonlinear heat equation, for which blowup may occur no matter how small $\alpha $ is, blowup  for~\eqref{NLS} cannot occur if $\alpha $ is too small.  

\begin{prop} \label{eNLS1} 
Suppose $0 <\alpha <\frac {4} {N}$ and let $\gamma \in \R$. It follows that for every $\DI \in H^1 (\R^N ) $, the  corresponding solution of~\eqref{NLS} is global, i.e. $\Tma =\infty $. 
\end{prop} 

\begin{proof} 
Let $\DI\in H^1 (\R^N ) $ and $u$ the corresponding solution of~\eqref{NLS} defined on the maximal interval $[0, \Tma)$. 
Formula~\eqref{fSL1} yields 
\begin{equation} \label{eNLS1:1} 
\| u(t) \| _{ L^2 } = e^{\gamma t}  \| \DI \| _{ L^2 } 
\end{equation} 
for all $0\le t<\Tma$. 
Applying the blowup alternative on the $L^2$ norm of Proposition~\ref{eLocGL2}, we conclude that $\Tma =\infty $. 
\end{proof} 

When $\alpha \ge \frac {4} {N}$, finite-time blowup may occur. This was proved in~\cite{Zakharov} in the three-dimensional cubic case with $\gamma =0$, then in~\cite{Glassey} in the general case (still with $\gamma =0$). 
Note that all solutions have locally bounded $L^2 $-norm by~\eqref{eNLS1:1}, so that Levine's method used in Section~\ref{sHeat} cannot be applied. 
Instead, the proof in~\cite{Zakharov, Glassey} is based on the variance identity~\eqref{fSL3}-\eqref{fSL4}. 
This argument can easily be applied to the case $\gamma \ge 0$, which we consider first.

\subsection{The case $\gamma \ge 0$} \label{sNLSGpos} 
The following result is proved in~\cite{Zakharov, Glassey} when $\gamma =0$.

\begin{thm} \label{eNLS2} 
Suppose $\frac {4} {N} \le \alpha < \frac {4} {N-2}$ and $\gamma \ge 0$. 
Let $\DI\in H^1 (\R^N ) $ and $u$ the corresponding solution of~\eqref{NLS}  defined on the maximal interval $[0, \Tma)$. If $\Ene (\DI) <0$ and $\DI \in   L^2 (\R^N ,  |x|^2dx)$, where $\Ene $ is defined by~\eqref{fDE1},  then $u$ blows up in finite time, i.e., $\Tma <\infty $.
\end{thm} 

\begin{proof} 
The proof is based on a differential inequality for the variance.
More precisely, it follows from~\eqref{fSL3} that
\begin{equation}  \label{fNLS15}
\frac {d} {dt} (e^{-2\gamma t} \Var (u(t)))= -4 e^{-2\gamma t} \Mom (u) 
\end{equation} 
and from~\eqref{fSL4} that
\begin{equation}  \label{fNLS14}
\frac {d} {dt} (e^{-2\gamma t} \Mom (u(t)))  =  e^{-2\gamma t}  \Bigl[ -4 \Ene (u(t))  + \frac { N\alpha -4 } {\alpha +2}  \|u \| _{ L^{\alpha +2} }^{\alpha +2}  \Bigr]   \ge -4 e^{-2\gamma t} \Ene (u(t))
\end{equation} 
where we used the assumption $N\alpha \ge 4$ in the last inequality. 
\eqref{fNLS15} and~\eqref{fNLS14} yield
\begin{equation}  \label{fNLS16}
\frac {d^2} {dt^2} (e^{-2\gamma t} \Var (u(t)))  = -4 \frac {d} {dt}( e^{-2\gamma t} \Mom (u(t)) )
 \le  16 e^{-2\gamma t}  \Ene (u(t)) .
\end{equation} 
Since
\begin{equation*}
\Neh (w)= (\alpha +2) \Ene (w) - \frac {\alpha } {2}\int  _{ \R^N  } |\nabla u|^2 \le (\alpha +2) \Ene (w) 
\end{equation*} 
and $\gamma \ge 0$, it follows from~\eqref{fSL2} that
\begin{equation} \label{fNLS19}
\frac {d} {dt} \Ene (u(t)) \le \gamma (\alpha +2) \Ene (u(t))
\end{equation} 
so that
\begin{equation} \label{fNLS20}
\Ene (u(t)) \le e^{\gamma (\alpha +2) t} \Ene ( \DI) <0.
\end{equation} 
Applying~\eqref{fNLS16}  and~\eqref{fNLS20}  we obtain
\begin{equation} \label{fNLS21}
\frac {d^2} {dt^2} (e^{-2\gamma t} \Var (u(t)))   \le  16 e^{\alpha \gamma t} \Ene ( \DI) \le 16 \Ene ( \DI ).
\end{equation} 
Note that by~\eqref{fNLS15} 
\begin{equation} \label{fNLS21b1} 
\frac {d} {dt} (e^{-2\gamma t} \Var (u(t)))   _{ |t=0 }=  -4 \Mom ( \DI )  .
\end{equation} 
Integrating twice~\eqref{fNLS21} and applying~\eqref{fNLS21b1} yields
\begin{equation} \label{fNLS22} 
e^{-2\gamma t} \Var (u(t))  \le \Var (\DI) -4 t   \Mom ( \DI )   + 16 \Ene ( \DI )\int _0^t \int _0^s e^{\alpha \gamma \sigma } \, d\sigma ds 
\end{equation} 
for all $0\le t< \Tma$. The right-hand side of~\eqref{fNLS22}, considered as a function of $t\ge 0$,  is negative for $t$ large (because $\Ene ( \DI)<0$). Since $e^{-2\gamma t} \Var (u(t))\ge 0$, we conclude that $\Tma <\infty $.
\end{proof} 

The ``natural" condition in Theorem~\ref{eNLS2} is $\Ene ( \DI )<0$. However, we require that $\DI \in L^2 (\R^N,  |x|^2 dx) $ because we calculate the variance $\Var (u)$. 
Whether the finite variance assumption is necessary or not in Theorem~\ref{eNLS2} seems to be an open question. 
A partial answer is known in the case $\alpha =\frac {4} {N}$ and $\gamma =0$: if $\Ene ( \DI )<0$ and $ \| \DI \| _{ L^2 }$ is not too large, then $\Tma <\infty $. (See~\cite[Theorem~1.1]{MerleR4}.) 
Another partial answer is given by Ogawa and Tsutsumi~\cite{OgawaTu} for radially symmetric solutions in the case $\gamma =0$ and $N\ge 2$. The proof can be adapted to the case $\gamma \ge 0$, under the additional restriction $N\ge 3$. (The case $N=1$, $\gamma =0$ and $\alpha =4$ is considered in~\cite{OgawaTd}, but we do not study its extension to $\gamma >0$ here.)

\begin{thm} \label{eNLS3} 
Suppose $\frac {4} {N} \le \alpha < \frac {4} {N-2}$ and $\gamma \ge 0$. 
Assume further $N\ge 2$ and $\alpha \le 4$ if $\gamma =0$, and $N\ge 3$ if $\gamma >0$.
Let $\DI\in H^1 (\R^N ) $ and $u$ the corresponding solution of~\eqref{NLS}  defined on the maximal interval $[0, \Tma)$. If $\Ene (\DI) <0$, where $\Ene $ is defined by~\eqref{fDE1}, and if $\DI$ is radially symmetric,  then $u$ blows up in finite time, i.e., $\Tma <\infty $.
\end{thm} 

\begin{proof} 
The proof uses calculations similar to those in the proof of Theorem~\ref{eNLS2}, but for a truncated variance. 
It is convenient to set $v(t) = e^{-\gamma t} u(t)$, so that $v$ satisfies the equation $v_t = i (\Delta v+ e^{\alpha \gamma t}  |v|^\alpha v)$. Moreover,
\begin{equation} \label{fNLS23b1} 
 \| v(t ) \| _{ L^2 } =  \| \DI \| _{ L^2 }
\end{equation} 
by formula~\eqref{eNLS1:1}.
Let $\GVar \in C^\infty (\R^N ) \cap W^{4, \infty } (\R^N ) $ be spherically symmetric and set
\begin{equation} \label{fNLS24} 
\zeta (t)=   \int  _{ \R^N  } \GVar  |v|^2 dx.
\end{equation} 
It follows from~\eqref{fVarug1} and~\eqref{fVaruqg2}  that
\begin{equation} \label{fOT1b1} 
\frac {1} {2} \zeta '(t)   =    \Im \int _{ \R^N  }  \overline{v} ( \nabla v \cdot  \nabla  \GVar) 
\end{equation}
and
\begin{equation} \label{fOT1} 
\frac {1} {2} \zeta '' =   \int  _{ \R^N  }  \Bigl( -\frac {1} {2}  |v|^2 \Delta ^2 \GVar - \frac {\alpha e^{\alpha  \gamma t}} {\alpha +2}  |v|^{\alpha +2} \Delta \GVar + 2  |\nabla v|^2 \GVar '' \Bigr) .
\end{equation} 
Note that the calculations in Proposition~\ref{eGGLzuu} are formal in the case $\theta =\frac {\pi } {2}$. However, they are easily justified for $H^2$ solutions, and then by continuous dependence for $H^1$ solutions. 
We observe that
\begin{multline*} 
 \int  _{ \R^N  }  \Bigl( - \frac {\alpha  e^{\alpha  \gamma t}} {\alpha +2}  |v|^{\alpha +2} \Delta \GVar + 2  |\nabla v|^2 \GVar '' \Bigr) = 2N\alpha  e^{-2\gamma t} \Ene (u) - (N\alpha -4) \int  _{ \R^N  }   |\nabla v|^2 \\
+ 2 \int  _{ \R^N  } ( \GVar '' -2)  |\nabla v|^2
  + \frac {\alpha e^{\alpha \gamma t}} {\alpha +2} \int   _{ \R^N  } (2N -  \Delta \GVar )  |v|^{\alpha +2}.
\end{multline*} 
Since $N\alpha \ge 4$ and $\Ene (u(t)) \le e^{\gamma (\alpha +2) t} \Ene ( \DI) $ by~\eqref{fNLS20}, we deduce that 
\begin{multline*} 
 \int  _{ \R^N  }  \Bigl( - \frac {\alpha e^{\alpha \gamma t}} {\alpha +2}  |v|^{\alpha +2} \Delta \GVar + 2  |\nabla v|^2 \GVar '' \Bigr)
\le  2N\alpha  e^{\alpha \gamma t}  \Ene ( \DI) \\+ 2 \int  _{ \R^N  } ( \GVar '' -2)  |\nabla v|^2 
  + \frac {\alpha e^{\alpha \gamma t}} {\alpha +2} \int   _{ \R^N  } (2N -  \Delta \GVar )  |v|^{\alpha +2}
\end{multline*} 
so that \eqref{fOT1}  yields
\begin{equation} \label{fOT2} 
\begin{split} 
\frac {1} {2} & \zeta '' \le   2N\alpha  e^{\alpha \gamma t} \Ene ( \DI)\\ & + \int  _{ \R^N  }  \Bigl( -\frac {1} {2}  |v|^2 \Delta ^2 \GVar + \frac {\alpha e^{\alpha \gamma t}} {\alpha +2}  |v|^{\alpha +2} (2N - \Delta \GVar ) + 2  |\nabla v|^2 (\GVar ''-2) \Bigr)  .
\end{split} 
\end{equation} 
We first consider the case $\gamma =0$. We apply Lemma~\ref{eOgau} with $A=  \| \DI \| _{ L^2 }$,  $\Avoir = \frac {\alpha } {\alpha +2}$ and $\varepsilon >0$ sufficiently small so that $ \Const    \Avoir  \varepsilon ^ {2(N-1) } <1$ and $\kappa ( \mu , \varepsilon  ) \le -N\alpha  \Ene ( \DI)$. With $\GVar = \GVar_\varepsilon $ given by Lemma~\ref{eOgau}, it follows from~\eqref{fNLS23b1}, \eqref{fOT2} and~\eqref{eOgau:u} that $\zeta '' \le   2N\alpha  \Ene ( \DI) <0$. Since $\zeta (t)\ge 0$ for all $0\le t<\Tma$, we conclude as in the proof of Theorem~\ref{eNLS2} that $\Tma <\infty $. 

We next consider the case $\gamma >0$ and $N\ge 3$. Let $0< \tau < \Tma$.
We set 
\begin{equation} \label{fOT3} 
\Avoir _\tau  = e^{\alpha \gamma \tau  } \ge 1
\end{equation} 
we fix 
\begin{equation} \label{fOT4} 
\frac {1} {2} > \lambda > \frac {1} {2(N-1)}
\end{equation} 
(here we use $N\ge 3$) and we set 
\begin{equation}  \label{fOT5} 
\varepsilon _\tau  = a    \Avoir _\tau  ^{-\lambda } \le a  . 
\end{equation} 
Here, the constant $0<a \le 1$ is 
chosen sufficiently small so that $ \Const   a  ^ {2(N-1) } <1$, where  $\Const  $ is the constant in Lemma~\ref{eOgau}. Since $\Avoir _\tau  \ge 1$ and $1-2(N-1)\lambda <0$ by~\eqref{fOT4}, it follows in particular that $ \Const    \Avoir  _\tau \varepsilon  _\tau ^ {2(N-1) }=  \Const    \Avoir _\tau ^{1- 2(N-1) \lambda } a   ^ {2(N-1) }    \le  \Const  a   ^ {2(N-1) }   <1$. Moreover, we deduce from~\eqref{fOT4} that $\kappa $ defined by~\eqref{eOgau:u1} satisfies $\kappa (\Avoir ,\varepsilon ) \le C \Avoir _\tau ^{1- \delta }$, where $C$ is independent of $\tau  $, and 
\begin{equation} \label{fDdel} 
\delta = \alpha \min  \Bigl\{  \frac {\lambda N } {2}, \frac {2(N-1)\lambda -1} {4-\alpha } \Bigr\} >0 .
\end{equation}  
We now let $\GVar = \GVar_{\varepsilon_\tau } $ where $\GVar_\varepsilon $ is given by Lemma~\ref{eOgau} 
for this choice of  $\varepsilon $. It follows from~\eqref{eOgau:1zb}, \eqref{eOgau:u}, and the inequality $\kappa (\Avoir ,\varepsilon ) \le C \Avoir _\tau ^{1- \delta }$ that
\begin{equation} \label{fOT6} 
\begin{split} 
 - 2 \int  _{ \R^N  } (2- \GVar_{\varepsilon_\tau } '' ) | \nabla v |^2  &
 + \frac {\alpha}{\alpha+2} e^{\alpha \gamma t} \int _{ \R^N  } (2N - \Delta  \GVar _{\varepsilon_\tau }) |v |^{\alpha +2} \\ &
  -\frac {1}{2} \int _{ \R^N  } |v |^2 \Delta^2  \GVar _{\varepsilon_\tau } \le  C \Avoir _\tau ^{1- \delta }.
\end{split} 
\end{equation} 
Estimates~\eqref{fOT2} and \eqref{fOT6} yield
 \begin{equation} \label{fOT7} 
 \frac {1} {2} \zeta ''  \le  2N \alpha  e^{\alpha \gamma t}   \Ene (\DI ) + C  \Avoir _\tau ^{1- \delta } 
\end{equation} 
for all $0\le t \le \tau $. 
Integrating twice~\eqref{fOT7} and applying~\eqref{fOT1b1}, we deduce that
 \begin{equation} \label{fOT10} 
\begin{split} 
 \frac {1} {2} \zeta (\tau ) &\le \frac {1} {2}  \|\GVar_{\varepsilon_\tau } \| _{ L^\infty  }    \|\DI \| _{ L^2 }^2  + \tau   \| \nabla \GVar_{\varepsilon_\tau } \| _{ L^\infty  }  \| \DI \| _{ H^1 }^2
 \\ & +  \frac {2N} {\alpha \gamma ^2}   \Ene (\DI ) (e^{\alpha \gamma \tau }- \tau  - \alpha \gamma \tau )
+  C \Avoir _\tau ^{1-\delta } \tau   . 
\end{split} 
\end{equation} 
Using~\eqref{eOgau:1b} to estimate $\GVar$ in the above inequality, applying~\eqref{fOT3} and~\eqref{fOT5} to express $\Avoir _\tau $ and $\varepsilon _\tau $ in terms of $\tau $, and since $\zeta (\tau )\ge 0$, we obtain
 \begin{equation} \label{fOT13} 
  0 \le C e^{2\lambda \alpha \gamma \tau }    + C \tau e^{ \lambda \alpha \gamma \tau }  
  +  \frac {2N} {\alpha \gamma ^2}   \Ene (\DI ) (e^{\alpha \gamma \tau }- \tau  - \alpha \gamma \tau )
+  C  \tau  e^{(1-\delta ) \alpha \gamma \tau } . 
\end{equation} 
Since $\Ene (\DI )<0$, and $\max\{ 2\lambda , 1-\delta  \}<1$, the right-hand side of~\eqref{fOT13} is negative for $\tau $ large. Since $\tau <\Tma$ is arbitrary, we conclude that $\Tma <\infty $.
\end{proof} 

\subsection{The case $\gamma < 0$} \label{sNLSGneg} 
If $\gamma <0$, the argument used in the proof of Theorem~\ref{eNLS2} breaks down because~\eqref{fNLS19} does not hold. Yet blowup occurs when $\alpha >\frac {4} {N}$, as shows the following result of Tsutsumi~\cite{Tsutsumi}.

\begin{thm} \label{eNLS2b} 
Suppose $\frac {4} {N} < \alpha < \frac {4} {N-2}$ and $\gamma < 0$. 
Let $\DI\in H^1 (\R^N ) $ and $u$ the corresponding solution of~\eqref{NLS}  defined on the maximal interval $[0, \Tma)$. If 
\begin{equation} \label{eNLS2b:1} 
\Var ( \DI ) + \frac {N\alpha -4} {\gamma \alpha } \Mom ( \DI ) + \frac {(N\alpha -4)^2} {\gamma ^2\alpha ^2} \Ene ( \DI ) <0
\end{equation} 
where the functionals $\Ene $, $\Var $ and $\Mom $ are defined by~\eqref{fDE1}, \eqref{fVar1} and~\eqref{fSL5}, respectively, then $u$ blows up in finite time, i.e., $\Tma <\infty $.
\end{thm} 

\begin{proof} 
We follow the simplified argument given in~\cite{OhtaT}.
We define
\begin{equation}  \label{fNLS23}
\Bca (w) = \frac {1} {2} \int  _{ \R^N  } |\nabla w|^2 - \frac {N\alpha } {4(\alpha +2)} \int  _{ \R^N  } |w|^{\alpha +2}
\end{equation} 
and we set ${e} (t)= \Ene (u(t))$, ${v}  (t)= \Var (u(t))$, ${\imath}  (t)= \Neh (u(t))$, $\jmath  (t)= \Mom (u(t))$, ${w}  (t)= \Bca (u(t))$, where the functionals $\Ene $, $\Var $, $\Neh $, $\Mom $ and $\Bca $ are defined by~\eqref{fDE1}, \eqref{fVar1}, \eqref{fDI1}, \eqref{fSL5}  and~\eqref{fNLS23}, respectively. 
It follows from~\eqref{fSL2}, \eqref{fSL3} and~\eqref{fSL4}  that
\begin{equation}  \label{fNLS30}
\frac {de} {dt}= \gamma   {\imath} (t) , \quad 
\frac {dv} {dt}= 2\gamma   {v}(t) - 4  {\jmath}  (t) , \quad 
\frac {d\jmath} {dt}= 2\gamma  {\jmath}  (t) -4  {w}(t) .
\end{equation} 
It is convenient to define
\begin{equation*} 
b = - 2 \gamma  \frac {4- (N-2) \alpha } {N\alpha -4} >0,\quad \eta = -2\gamma +b = \frac {-4\gamma \alpha } {N\alpha -4} >0.
\end{equation*} 
Using the identity $\gamma \imath  - be  = -\eta w$, we deduce from~\eqref{fNLS30} that
\begin{gather}  
\frac {d} {dt} (e^{-bt}  {e} (t))= e^{-bt} ( \gamma  { \imath } - b  {e} )  \label{fNLS36}  = - \eta  e^{-bt}  {w} (t) \\
\frac {d} {dt} (e^{-bt}  {v } (t))= - \eta  e^{-bt}  {v} (t) - 4  e^{-bt}  { \jmath } (t)  \label{fNLS37} \\
\frac {d} {dt} (e^{-bt}  { \jmath } (t))= - \eta  e^{-bt}  { \jmath } (t) - 4  e^{-bt}  { w } (t). \label{fNLS38}
\end{gather} 
Integrating~\eqref{fNLS36} on $(0,t)$, we obtain
\begin{equation*}  
e^{-bt}  {e} (t) + \eta  \int _0^t e^{-bs}  {w} (t) = \Ene  (\DI ).
\end{equation*} 
Since $\alpha \ge \frac {4} {N}$, we have $ {e} \ge  {w} $ so that
\begin{equation}   \label{fNLS40}  
e^{-bt}  {w} (t) + \eta  \int _0^t e^{-bs}  {w} (s)\, ds \le  \Ene  (\DI ) .
\end{equation} 
We now set
\begin{equation*} 
 \widetilde{w }(t)=  \int _0^t e^{-bs}  \widetilde{w} (s)\, ds, \quad  \widetilde{ \jmath  }(t)=  \int _0^t e^{-bs}  \widetilde{\jmath} (s)\, ds
\end{equation*} 
so that~\eqref{fNLS40} becomes $ \frac {d \widetilde{w} } {dt} + \eta  \widetilde{w }  \le  \Ene  (\DI )$. 
Therefore, $e^{\eta t} \widetilde{w }(t) \le \frac {e^{\eta t}- 1} {\eta } \Ene ( \DI )$, which implies 
\begin{equation} \label{fNLS45}
\int _0^t e^{\eta s} \widetilde{w }(s) \, ds \le \frac {e^{\eta t}- 1 - \eta t} {\eta ^2} \Ene ( \DI ). 
\end{equation} 
Integrating now~\eqref{fNLS38} on $(0,t)$, we obtain
$\frac {d\widetilde{\jmath } } {dt} + \eta  \widetilde{\jmath } = \Mom ( \DI ) -4  \widetilde{w } $, so that
\begin{equation}   \label{fNLS47}  
 e^{\eta t} \widetilde{\jmath }  (t) = \int _0^t e^{\eta s} [ \Mom ( \DI ) -4  \widetilde{w } (s)]\, ds .
\end{equation} 
We deduce from~\eqref{fNLS47} and~\eqref{fNLS45} that
\begin{equation}   \label{fNLS48}  
 e^{\eta t} \widetilde{\jmath }  (t) \ge   \frac {e^{\eta t}- 1} {\eta } \Mom  ( \DI ) - 4 \frac {e^{\eta t}- 1 - \eta t} {\eta ^2} \Ene ( \DI ).
\end{equation} 
Finally, since $ {v} (t)\ge 0$ we deduce from~\eqref{fNLS37} that 
$\frac {d} {dt} (e^{-bt}  {v } (t))\le  - 4  e^{-bt}  { \jmath } (t)$, so that
\begin{equation}   \label{fNLS50}
e^{-bt}  {v } (t)\le \Var (\DI )  - 4     \widetilde{ \jmath  } (t).
\end{equation} 
It now follows from~\eqref{fNLS50} and~\eqref{fNLS48} that
\begin{equation}   \label{fNLS51}
e^{-bt}  {v } (t)\le \Var (\DI )  - 4     \frac {1- e^{- \eta t} } {\eta } \Mom  ( \DI ) +16 \frac {1- (1 - \eta t) e^{- \eta t}} {\eta ^2} \Ene ( \DI ).
\end{equation} 
Assumption~\eqref{eNLS2b:1} means that $\Var (\DI )  -      \frac {4} {\eta } \Mom  ( \DI ) + \frac {16} {\eta ^2} \Ene ( \DI ) <0 $. Therefore, the right-hand side of~\eqref{fNLS51} becomes negative for $t$ large, which implies that $\Tma <\infty $.
\end{proof} 

\begin{rem} Here are a few comments on Theorem~\ref{eNLS2b}.
\begin{enumerate}[{\rm (i)}] 

\item The condition~\eqref{eNLS2b:1} is satisfied if $\DI = c \varphi $ with $\varphi \in H^1 (\R^N ) $, $\varphi \not = 0$ and $c$ is large. 

\item The condition $\alpha >\frac {4} {N}$ is essential in the proof, for the definition of $\eta$ and $b$.  If $\alpha =4/N$, then finite-time blowup occurs for some initial data~\cite{Mohamad, Correia}, but the proof follows  a very different argument.

\item We are not aware of a result similar to Theorem~\ref{eNLS2b} for initial values of infinite variance (in the spirit of Theorem~\ref{eNLS3}).

\end{enumerate} 
\end{rem} 

\section{The complex Ginzburg-Landau equation} \label{sGL} 

\subsection{Sufficient condition for finite-time blowup}  \label{sBlowup} 
In this section, we derive sufficient conditions for finite-time blowup in equation~\eqref{GL}, and upper estimates of the blowup time. Such conditions are obtained in~\cite{SnoussiT} in the case $\gamma \le 0$ (with extra conditions on the parameters), in~\cite{CazenaveDW} in the case $\gamma =0$, and in~\cite{CazenaveDF} in the case $\gamma >0$. The 
upper bound is established in~\cite{CazenaveDW} in the case $\gamma =0$.

\begin{thm} \label{eBUg1} 
Let $\gamma \in \R$, $\alpha >0$, $0\le \theta <\frac {\pi } {2}$, $\DI \in \Cz \cap H^1 (\R^N ) $, and let $u$ be the corresponding solution of~\eqref{GL} defined on the maximal interval $[0, \Tma)$.
If
\begin{equation} \label{eBUg1:1}
\begin{cases} 
\Ene (\DI) <0 & \gamma \ge 0 \\
\Ene  _{ \frac {\gamma } {\cos \theta } } (\DI) <0 & \gamma \le 0
\end{cases} 
\end{equation} 
(with the definitions~\eqref{fDE1} and~\eqref{fHeat4}) then  $u$ blows up in finite time, i.e., $\Tma <\infty $.
Moreover,
\begin{equation} \label{eBUg1:2}
\Tma \le 
\begin{cases} 
\displaystyle \frac {1} { \gamma \alpha} \log  \Bigl( 1+  \frac { \gamma  \| \DI \| _{ L^2 }^2} { (\alpha +2)    (- \Ene (\DI ))  \cos \theta } \Bigr)  & \gamma >0 \\ \displaystyle \frac { \| \DI \| _{ L^2 }^2} { \alpha (\alpha +2)    (- \Ene (\DI ))  \cos \theta }  & \gamma =0 \\ 
\displaystyle \frac {1} {- \gamma \alpha} \log  \Bigl( 1+  \frac {-2 \gamma  \| \DI \| _{ L^2 }^2} { 2(\alpha +2)    (- \Ene  _{ \frac {\gamma } {\cos \theta } }(\DI ))  \cos \theta  - \gamma \alpha  \| \DI \| _{ L^2 }^2} \Bigr)
& \gamma <0 .
\end{cases} 
\end{equation}
\end{thm} 

\begin{proof} 
We consider separately the cases $\gamma \ge 0$ and $\gamma <0$.

\medskip 

\noindent  {\bf The case $\gamma \ge 0$.}\quad 
We follow the argument of the proof of Theorem~\ref{eNLH4}, and in particular we use the same notation~\eqref{fNot1}.
We only indicate the minor changes that are necessary. 
The function $v(t)= e^{- \gamma t} u(t)$ now satisfies the equation 
\begin{equation} \label{GL2b}
\begin{cases} 
v_t= e^{i\theta } [\Delta v+ e^{\alpha \gamma t} |v|^\alpha v] \\
v(0)= \DI.
\end{cases} 
\end{equation}  
Identity~\eqref{eNLH4:3} becomes
\begin{equation} \label{fPos6b}
\int  _{ \R^N  }  \overline{v}v_t= -e^{i\theta }   \widetilde{\jmath} (t),
\end{equation}
so that
\begin{equation} \label{fPos6}
\frac {d  \widetilde{f}} {dt} = 2 \Re \int _{\R^N }  \overline{v}  v_t =   - 2 \widetilde{\jmath } (t)  \cos \theta.
\end{equation}
Moreover, $\frac {d  \widetilde{e} } {dt}   =  -  \cos \theta   \int  _{ \R^N  } |v_t|^2  + \alpha \gamma  \widetilde{e}  -\frac {\alpha \gamma } {2} \int  _{ \R^N  } |\nabla u|^2$, so that inequality~\eqref{eNLH4:5} becomes
\begin{equation} \label{fPos7}
\frac {d  \widetilde{e} } {dt}  -\alpha \gamma  \widetilde{e}  \le   - \cos \theta   \int  _{ \R^N  } |v_t|^2.
\end{equation}
 Applying~\eqref{fPos7}, Cauchy--Schwarz,  \eqref{fPos6b} and~\eqref{fPos6} we obtain
\begin{equation} \label{fPos8}
-  \widetilde{f}    \Bigl(    \frac {d  \widetilde{e} } {dt} -\alpha \gamma  \widetilde{e}   \Bigr) \ge   \cos \theta   \int  |v|^2 \int  |v_t|^2 \ge   \cos \theta   \Bigl| \int  \overline{v}v_t  \Bigr|^2   =     \cos \theta   \widetilde{\jmath }  ^2  = \frac {1} {2}   (-  \widetilde{\jmath }   ) \frac {d \widetilde{f} } {dt} .
\end{equation}
The crux is that the factor $\cos \theta $ in the first inequalities in~\eqref{fPos8} has been cancelled in the last one by using~\eqref{fPos6}. In particular, the left-hand and the right-hand terms in~\eqref{fPos8} are the same as in~\eqref{eNLH4:7}. 
Therefore, we may now continue the argument as in the proof of Theorem~\ref{eNLH4}. Using~\eqref{fPos6} instead of~\eqref{eNLH4:3}, we arrive at the inequality
\begin{equation*}
\alpha  (\alpha +2)   (- \Ene (\DI) ) \|\DI\| _{ L^2 } ^{-( \alpha +2) }
e^{  \alpha \gamma t } \cos \theta  +  \frac {d} {dt} ( \widetilde{f}^{- \frac {\alpha } {2}} ) \le 0
\end{equation*}
and estimate~\eqref{eBUg1:2} easily follows in both the cases $\gamma >0$ and $\gamma =0$.

\medskip 

\noindent  {\bf The case $\gamma < 0$.}\quad 
Since the result in the case $\gamma \ge 0$ is obtained by the argument of the proof of Theorem~\ref{eNLH4}, one could try now to follow the proof of Theorem~\ref{eNLH2}. It turns out that this strategy leads to intricate calculations and unnecessary conditions. (See~\cite{CazenaveDF}.) Instead, we follow the strategy of~\cite{SnoussiT}
and we  set
\begin{align}
\mu & = (-\gamma )^{-\frac {1} {2}} (\cos \theta )^{\frac {1} {2}}  \label{fSC3} \\
v(t,x) &= e^{- i t \sin \theta }  \mu ^{\frac {2} {\alpha }}  u  (\mu ^2 t, \mu x)  \label{fSC1} \\
\DIb (x) &=  \mu ^{\frac {2} {\alpha }}  \DI (\mu x) \label{fSC2}
\end{align}
so that
\begin{equation} \label{GL2}
\begin{cases}
\displaystyle  v_t =  e^{ i\theta } [  \Delta v +  |v|^\alpha v -v ]  \\
v(0 )= \DIb.
\end{cases}
\end{equation}
Since $u$ is defined on $[0, \Tma )$,  $v$ is defined on $[0, \Sma )$ with 
\begin{equation} \label{fNeg4} 
\Sma = \frac {-\gamma \Tma} {\cos \theta }. 
\end{equation} 
We introduce the notation
\begin{equation} \label{fNot2} 
 \widetilde{f}  =  \| v  \| _{ L^2 }^2 , \quad   
 \widetilde{\jmath} =  \Neh _{ -1 } (v(t))  ,    \quad  
  \widetilde{e} = \Ene _{ -1 } (v(t))
\end{equation}
where $\Neh _{ -1 }$ and $\Ene _{ -1 }$ are defined by~\eqref{fHeat3} and~\eqref{fHeat4}, and we observe that
\begin{equation} \label{fNeg2} 
 \| \DIb \| _{ L^2 }^2 = \mu ^{ \frac {4} {\alpha } - N}   \| \DI \| _{ L^2 }^2, \quad 
\Ene  _{ -1 } ( \DIb ) = \mu ^{2 + \frac {4} {\alpha } - N} \Ene  _{ \frac {\gamma } {\cos \theta } } ( \DI ).
\end{equation} 
We now follow the proof of Theorem~\ref{eNLH2}. 
Equation~\eqref{GL2} yields
\begin{gather}
  \int  _{ \R^N  }  \overline{v} v_t   = - e^{i\theta }  \widetilde{\jmath}  (t)  \label{fSN2} \\
   \frac {d  \widetilde{f} } {dt}  = - 2   \widetilde{\jmath} (t)  \cos \theta \label{fSN1}  \\
\frac {d  \widetilde{e}} {dt} = - \cos \theta \int  _{ \R^N  }  |v_t|^2 . \label{fSN3}
\end{gather}
Since $\Ene  _{ \frac {\gamma } {\cos \theta } } ( \DI )<0$, we deduce from~\eqref{fNeg2} that $ \widetilde{e} (0)<0$. Therefore, $ \widetilde{e} (t) <0 $ by~\eqref{fSN3} (hence $ \widetilde{\jmath} (t) <0$) and $\frac {d  \widetilde{f} } {dt} >0$ by~\eqref{fSN1}.
Applying~\eqref{fSN3}, Cauchy--Schwarz,  \eqref{fSN2} and~\eqref{fSN1} we obtain
\begin{equation} \label{fNeg1}
-   \widetilde{f}   \frac {d   \widetilde{e}  } {dt}  =   \cos \theta   \int  |v|^2 \int  |v_t|^2 \ge   \cos \theta   \Bigl| \int  \overline{v}v_t  \Bigr|^2   =     \cos \theta   \widetilde{\jmath }  ^2  = \frac {1} {2}   (-  \widetilde{\jmath }   ) \frac {d \widetilde{f} } {dt} .
\end{equation}
At this point, we use the property
\begin{equation}\label{fNeg5b1} 
 \widetilde{j} (t)  = (\alpha +2)  \widetilde{e} (t) -  \frac {\alpha   } {2  }  \widetilde{f} (t)  - \frac {\alpha } {2 }\int  _{ \R^N  } |\nabla v(t)|^2
 \le (\alpha +2)  \widetilde{e} (t) -  \frac {\alpha   } {2  }  \widetilde{f} (t) <0
\end{equation}
so that~\eqref{fNeg1} yields $-   \widetilde{f}   \frac {d   \widetilde{e}  } {dt}  \le  \frac {1} {2}   (- (\alpha +2)  \widetilde{e} +  \frac {\alpha   } {2  }  \widetilde{f}  ) \frac {d \widetilde{f} } {dt} $.
Therefore, $\frac {d} {dt}  ( - \widetilde{e}   \widetilde{f} ^{-\frac {\alpha +2} {2}} + \frac {1 } {2 }  \widetilde{f} ^{-\frac {\alpha } {2}} )  \ge 0$,
and so
\begin{equation} \label{fNeg5}
- \widetilde{e}  + \frac {1 } {2 }  \widetilde{f}   \ge \eta  \widetilde{f} ^{\frac {\alpha +2} {2}} ,
\end{equation}
with
\begin{equation} \label{fNeg6}
\eta = ( -  \widetilde{e} (0) )  \widetilde{f} (0) ^{- \frac {\alpha +2} {2}} + \frac {1 } {2 }  \widetilde{f} (0)^{-\frac {\alpha } {2} } >0.
\end{equation}
It follows from~\eqref{fSN1}, \eqref{fNeg5b1}, and~\eqref{fNeg5} that
\begin{equation*}
\frac {d \widetilde{f}} {dt} \ge [ 2(\alpha +2)   (- \widetilde{e} ) + \alpha  \widetilde{f}  ] \cos \theta \ge  (-2   \widetilde{f}  +  2(\alpha +2) \eta   \widetilde{f} ^{\frac {\alpha +2} {2}})  \cos \theta .
\end{equation*}
Therefore,
$ \frac {d} {dt}  [ (e^{ 2 t \cos \theta }  \widetilde{f} )^{-\frac {\alpha } {2}} - (\alpha +2) \eta  e ^{ - \alpha  t \cos \theta } ] \le 0$, so that $ (\alpha +2) \eta  (1- e^{ -\alpha  t \cos \theta } ) \le    \widetilde{f} (0)^{-\frac {\alpha } {2}}$
for all $0\le t<\Sma$. It follows easily that
\begin{equation} \label{fES1} 
\Sma \le  \frac {1} {\alpha \cos \theta } \log  \Bigl(  1+ \frac{2  \| \DIb \| _{ L^2 }^2 } { 2(\alpha +2) (-  \Ene _{ -1 } ( \DIb )  +  \alpha  \| \DIb \| _{ L^2 }^2 } \Bigr) .
\end{equation}
Applying~\eqref{fNeg4}, \eqref{fNeg2}, and~\eqref{fSC3}, estimate~\eqref{eBUg1:2}  follows.
\end{proof} 

\begin{rem} \label{eGN1}
One can study equation~\eqref{GL2} for its own sake. The proof of Theorem~\ref{eBUg1} shows that if $\DIb \in \Cz \cap H^1 (\R^N ) $ satisfies $\Ene _{ -1 }  (\DIb ) \le 0$ and $\DIb \not \equiv 0$, then the corresponding solution of~\eqref{GL2} defined on the maximal interval $[0, \Sma)$ blows up in finite time, and~\eqref{fES1} holds.
It follows from~\eqref{fES1} that
\begin{equation} \label{fES1b3}
\Sma \le  \frac {1} {\alpha \cos \theta } \log (\frac{\alpha+2}{\alpha}).
\end{equation}
In particular, the bound in~\eqref{fES1b3} is independent of $\DIb$, so that this result does not apply to solutions
for which the blow-up time would be large. When $\alpha <\frac {4} {N-2}$, this restriction can be improved by  the potential well argument we used in Theorem~\ref{ePSb3} for the heat equation. 
More precisely, if  $\DIb \in \Cz \cap H^1 (\R^N ) $ satisfies $ \Ene _{ -1 } (\DIb ) < \Ene _{ -1 }( \Gst _{ -1 })$ and $\Neh _{ -1 }( \DIb )<0$, where $ \Gst _{ -1 }$ is as in Theorem~\ref{ePSb3}, then the corresponding solution of~\eqref{GL2} defined on the maximal interval $[0, \Sma)$.
blows up in finite time, i.e., $\Sma <\infty $, and
\begin{equation} \label{eGN3:1}
\Sma \le \frac {1} {\alpha \cos \theta } \left[ \frac {(\alpha +4) [\Ene _{ -1 }( \DIb )]^+} {\Ene _{ -1 }( \Gst _{ -1 }) - \Ene _{ -1 }( \DIb )} +   \log  \Bigl(  \frac {2(\alpha +2)} {\alpha } \Bigr) \right].
\end{equation}
The proof is easily adapted from the proof of Theorem~\ref{ePSb3}, in the same way as the proof of Theorem~\ref{eBUg1} (case $\gamma \le 0$) is adapted from the proof of Theorem~\ref{eNLH2}.
Note that this last result applies to solutions for which the maximal existence time is arbitrary large.
Indeed, given $\varepsilon >0$,  $\DIb ^\varepsilon = (1+\varepsilon )  \Gst _{ -1 }$ satisfies $\Ene _{ -1 }(\DIb ^\varepsilon)<\Ene _{ -1 }( \Gst _{ -1 })$ and $\Neh _{ -1 }(\DIb ^\varepsilon)<0$, while the blowup time of the corresponding solution of~\eqref{GL2} goes to infinity as $\varepsilon \downarrow 0$. (Cf. Remark~\ref{ePSb4}.)
\end{rem}

\begin{rem} \label{eBUg2} 
Here are some comments on Theorem~\ref{eBUg1}.
\begin{enumerate}[{\rm (i)}] 

\item \label{eBUg2:1} 
If $\gamma \le 0$, one can replace assumption~\eqref{eBUg1:1} by the slightly weaker assumption $\Ene  _{ \frac {\gamma } {\cos \theta } }(\DI) \le 0$ and $\DI \not = 0$. 
See Remark~\ref{eNLH3}~\eqref{eNLH3:0}. 

\item \label{eBUg2:2} Let $\alpha >0$, $ 0\le \theta < \frac {\pi } {2}$, and fix $\DI  \in \Cz \cap H^1 (\R^N )$ such that $\Ene ( \DI) <0$. 
It follows from~\eqref{eBUg1:2} that $\Tma \to 0$ as $\gamma \to \infty $. 
On the other hand, it is clear that if $\gamma $ is sufficiently negative, then $\Ene  _{ \frac {\gamma } {\cos \theta } } ( \DI) \ge  0$, so that one cannot apply Theorem~\ref{eBUg1}. This is not surprising,  since the corresponding solution of~\eqref{GL} is global if $\gamma $ is sufficiently negative. (See Remark~\ref{eLocGL3}.)

\item \label{eBUg2:3} Let $\alpha >0$ and $\gamma <0$. Given $\DI \in \Cz \cap H^1 (\R^N ) $ such that $\Ene  _{ \frac {\gamma } {\cos \theta } } ( \DI )<0$, it follows from~\eqref{eBUg1:2} that $\Tma < \frac {1} {-\gamma \alpha } \log (\frac {\alpha +2} {\alpha })$. In particular, we see that Theorem~\ref{eBUg1} does not apply to solutions
for which the blow-up time would be large. However, one can use the result presented in Remark~\ref{eGN1} above. Using the transformation~\eqref{fSC3}--\eqref{fSC2}, and formulas~\eqref{fNeg4}  and~\eqref{fNeg2}, we deduce from~\eqref{eGN3:1}  that if $ \Ene  _{ \frac {\gamma } {\cos \theta } } ( \DI ) < \mu ^{-2- \frac {4} {\alpha }+ N} \Ene _{ -1 }( \Gst _{ -1 })$ and $ \Neh  _{ \frac {\gamma } {\cos \theta } } ( \DI ) < 0$, then the corresponding solution of~\eqref{GL} blows up in finite time and 
\begin{equation*} 
\Tma \le \frac {1} {-\gamma \alpha } \left[ \frac {(\alpha +4) [\Ene _{ \frac {\gamma } {\cos \theta } }( \DI )]^+} { \mu ^{-2- \frac {4} {\alpha }+ N}\Ene _{ -1 }( \Gst _{ -1 }) -  \Ene _{ \frac {\gamma } {\cos \theta } }( \DI )} +   \log  \Bigl(  \frac {2(\alpha +2)} {\alpha } \Bigr) \right].
\end{equation*} 
Moreover, this property applies to solutions for which the maximal existence time is arbitrary large. (The stationary solution $ \Gst _{ -1 }$ of~\eqref{GL2} corresponds to the standing wave $u(t,x) = \mu ^{- \frac {2} {\alpha }} e^{it \mu ^{-2} \sin \theta } \Gst _{ -1 } (\mu ^{-1} x) $ of~\eqref{GL}.)
\end{enumerate} 
\end{rem} 

\subsection{Behavior of the blowup time as a function of $\theta $} \label{sBUT} 

Fix $\alpha >0$ and $\gamma \in \R$. Given $\DI\in \Cz \cap H^1 (\R^N ) $, we let $u^\theta $, for $ 0\le \theta < \frac {\pi } {2}$, be the solution of~\eqref{GL} defined on the maximal interval $[0, \Tma^\theta )$. 
If $\gamma \ge 0$ and $\Ene (\DI) <0$, or if $\gamma <0$ and $\Ene  _{ \frac {\gamma } {\cos \theta } } (\DI) <0$, then we know that $\Tma^\theta <\infty $. (See Theorem~\ref{eBUg1}.)
We now study the behavior of $\Tma^\theta $ as $\theta \to  \frac {\pi } {2}$, i.e. as the equation~\eqref{GL} approaches the nonlinear Schr\"o\-din\-ger equation~\eqref{NLS}. 
We consider separately the cases $\alpha <\frac {4} {N}$ and $\alpha \ge \frac {4} {N}$. 

\subsubsection{The case $\alpha <\frac {4} {N}$} \label{sGP}

If $\alpha <\frac {4} {N}$, then all solutions of the limiting equation~\eqref{NLS}  are global by Proposition~\ref{eNLS1}, and so we may expect that $\Tma^\theta \to \infty $ as $\theta \to \frac {\pi } {2}$. 
This is indeed what happens. 
Indeed, one possible proof of global existence for the nonlinear Schr\"o\-din\-ger equation~\eqref{NLS} is based on the Gagliardo-Nirenberg inequality
\begin{equation} \label{fGNb}
 \| w \| _{ L^{\alpha +2} }^{\alpha +2} \le   \| \nabla w \| _{ L^2 }^2 + A  \|w\| _{ L^2 }^{2 + \frac {4\alpha } {4-N\alpha }},
\end{equation}
where $A$ is a constant independent of $w\in H^1 (\R^N ) $. (See eg.~\cite{AdamsF}.)
Similarly, using~\eqref{fGNb} one can prove the following result. 
(For $\gamma =0$, this is~\cite[Theorem~1.2]{CazenaveDW}.)

\begin{thm} \label{eGlob1} 
Let $0<\alpha <\frac {4} {N}$ and $\gamma \in \R$. Given $\DI\in \Cz \cap H^1 (\R^N ) $, let $u^\theta $, for $ 0\le \theta < \frac {\pi } {2}$, be the solution of~\eqref{GL} defined on the maximal interval $[0, \Tma^\theta )$. 
If $\gamma \ge 0$, then
\begin{equation}  \label{eGlob1:1} 
\Tma ^\theta \ge 
\begin{cases} 
\displaystyle \frac {4-N\alpha } {4\alpha \gamma } \log  \Bigl( 1+ \frac {\gamma } {A  \| \DI \| _{ L^2 }^{\frac {4\alpha } {4-N\alpha }} \cos \theta} \Bigr) & \gamma >0 \\
\displaystyle \frac {4-N\alpha } {4\alpha A  \| \DI \| _{ L^2 }^{\frac {4\alpha } {4-N\alpha }} \cos \theta } & \gamma =0 
\end{cases} 
\end{equation} 
and if $\gamma <0$, then
\begin{equation}  \label{eGlob1:2} 
\Tma ^\theta \ge 
\begin{cases} 
 \infty  & \displaystyle  \cos \theta \le \frac {-\gamma } {A  \| \DI \| _{ L^2 }^{\frac {4\alpha } {4-N\alpha }}} \\
\displaystyle   \frac {4-N\alpha } {4\alpha \gamma } \log  \Bigl( 1+ \frac {\gamma } {A  \| \DI \| _{ L^2 }^{\frac {4\alpha } {4-N\alpha }} \cos \theta} \Bigr) & \displaystyle  \cos \theta > \frac {-\gamma } {A  \| \DI \| _{ L^2 }^{\frac {4\alpha } {4-N\alpha }}} 
\end{cases} 
\end{equation} 
where $A$ is the constant in~\eqref{fGNb}.
\end{thm} 

\begin{proof} 
We combine~\eqref{fGGLs} and~\eqref{fGNb} to obtain the desired conclusion. Setting $f(t)=  \| u^\theta (t)\| _{ L^2 }^2$, we deduce from~\eqref{fGGLs} and~\eqref{fGNb} that
\begin{equation*} 
\frac {df} {dt} = 2\gamma f + \cos \theta  ( -2  \|\nabla u\| _{ L^2 }^2 + 2 \| w \| _{ L^{\alpha +2} }^{\alpha +2})  \le  2\gamma f + 2 A  f^{1 + \frac {2\alpha } {4-N\alpha }} \cos \theta 
\end{equation*} 
so that
\begin{equation*} 
\frac {d} {dt} (e^{-2\gamma t} f) \le 2A e^{\frac {4\alpha \gamma t} {4-N\alpha }} (e^{-2\gamma t} f) ^{1 + \frac {2\alpha } {4-N\alpha }} \cos \theta .
\end{equation*} 
This means that $\frac {d} {dt} (- (e^{-2\gamma t} f) ^{-\frac {2\alpha } {4-N\alpha }}) \le \frac {4A\alpha \cos \theta } {4-N\alpha } e^{\frac {4\alpha \gamma t} {4-N\alpha }}$, from which we deduce that
\begin{equation*} 
(e^{-2\gamma t} f) ^{-\frac {2\alpha } {4-N\alpha }} \ge  \| \DI \|^{-\frac {4\alpha } {4-N\alpha }} - \frac{4A\alpha \cos \theta } {4-N\alpha } \int _0^t e^{\frac {4\alpha \gamma s} {4-N\alpha }} ds.
\end{equation*} 
The above inequality yields a control on $ \|u(t)\| _{ L^2 }$ for all $0\le t< \Tma^\theta $ such that $\frac{4A\alpha \cos \theta } {4-N\alpha } \int _0^t e^{\frac {4\alpha \gamma s} {4-N\alpha }} ds <  \| \DI \|^{-\frac {4\alpha } {4-N\alpha }}$. Therefore, if we set
\begin{equation}  \label{eGlob1:3} 
\tau = \sup  \Bigl\{ 0\le t<\infty ;\,  \frac{4A\alpha \cos \theta } {4-N\alpha } \int _0^t e^{\frac {4\alpha \gamma s} {4-N\alpha }} ds <  \| \DI \|^{-\frac {4\alpha } {4-N\alpha }} \Bigr\}
\end{equation} 
then it follows from the blowup alternative on the $L^2$-norm in Proposition~\ref{eLocGL1} that
$\Tma ^\theta \ge \tau $.
The result follows by calculating the integral in~\eqref{eGlob1:3} in the various cases $\gamma >0$, $\gamma =0$ and $\gamma <0$.
\end{proof} 

\begin{rem} \label{eRem1} 
Here are some comments on Theorem~\ref{eGlob1}.
\begin{enumerate}[{\rm (i)}] 

\item \label{eRem1:1} 
Suppose $\gamma \ge 0$. 
It follows from~\eqref{eGlob1:2} that the upper estimate~\eqref{eBUg1:2}  of $\Tma^\theta $ established in Theorem~\ref{eBUg1} is optimal with respect to the dependence in $\theta $. Indeed, if $\Ene ( \DI ) <0$, then it follows from~\eqref{eGlob1:2} and~\eqref{eBUg1:2} that 
\begin{equation*} 
0 < \liminf  _{ \theta \to \frac {\pi } {2} } \, \phi ( \theta ) \Tma^\theta  \le  \limsup  _{ \theta \to \frac {\pi } {2} }  \, \phi (\theta ) \Tma^\theta  <\infty 
\end{equation*} 
with $\phi (\theta ) = \cos \theta $ for $\gamma =0$ and $\phi (\theta )= [\log ((\cos\theta )^{-1})]^{-1}$ for $\gamma >0$. 

\item \label{eRem1:2} 
Suppose $\gamma <0$. It follows from~\eqref{eGlob1:2} that, given any initial value $\DI\in \Cz \cap H^1 (\R^N ) $, the corresponding solution of~\eqref{GL} is global for all $\theta $ sufficiently close to $\frac {\pi } {2}$.

\item \label{eRem1:3} 
We can apply Theorem~\ref{eGlob1} to equation~\eqref{GL2}.
The upper estimate~\eqref{eGlob1:2},  together with formulas~\eqref{fNeg2} and~\eqref{fNeg4}, shows that if
 $\DIb  \in \Cz \cap H^1 (\R^N ) $ and $v$ is the corresponding solution of~\eqref{GL2} defined on the maximal interval $[0, \Sma )$, then $\Sma = \infty $ if $A   \| \DIb \| _{ L^2 }^{ \frac { 4\alpha } {4-N\alpha } } \le 1$,
and 
\begin{equation} \label{fETM2}
\Sma \ge  - \frac {4-N\alpha} {4\alpha \cos \theta  } \log \Bigl( 1-  \frac { 1  } { A   \| \DIb \| _{ L^2 }^{ \frac { 4\alpha } {4-N\alpha } } } \Bigr) 
\end{equation}
if $A   \| \DIb \| _{ L^2 }^{ \frac { 4\alpha } {4-N\alpha } } > 1$.
Suppose now that  either $ \Ene _{ -1 }  ( \DIb ) \le 0$ and $\DIb \not = 0$ or else $0<  \Ene _{ -1 } ( \DIb ) <  \Ene _{ -1 } (  \Gst _{ -1 } ) $ and $ \Neh _{ -1 } ( \DIb )<0$.
In both cases $ \Neh _{ -1 } ( \DIb )<0$, and it follows from~\eqref{fGNb}  that
\begin{equation*}
 \| \nabla \DIb \| _{ L^2 }^2 + \|  \DIb \| _{ L^2 }^2 <  \| \DIb \| _{ L^{\alpha +2} }^{\alpha +2}
\le \| \nabla \DIb \| _{ L^2 }^2 + A   \| \DIb \| _{ L^2 }^{2 + \frac {4\alpha } {4-N\alpha }} .
\end{equation*}
In particular, $A   \| \DIb \| _{ L^2 }^{ \frac { 4\alpha } {4-N\alpha } } > 1$ so that
\begin{equation*} 
0 < \liminf  _{ \theta \to \frac {\pi } {2} } \, ( \cos \theta )\Sma^\theta \le  \limsup  _{ \theta \to \frac {\pi } {2} } \,  ( \cos \theta )  \Sma^\theta  <\infty 
\end{equation*} 
by~\eqref{fETM2}, and either~\eqref{fES1} or~\eqref{eGN3:1}.

\end{enumerate} 
\end{rem} 

\subsubsection{The case  $\alpha >\frac {4} {N}$} \label{sGN}

If $ \frac {4} {N} \le \alpha < \frac {4} {N-2}$, then the solution of the limiting nonlinear Schr\"o\-din\-ger equation~\eqref{NLS} blows up in finite time, under appropriate assumptions on the initial value $\DI$. See Theorems~\ref{eNLS2}, \ref{eNLS3} and~\ref{eNLS2b}.
Under these assumptions, one might expect that $\Tma ^\theta (\DI )$, which is finite
(under suitable assumptions) by Theorem~\ref{eBUg1:1}, remains bounded as  $\theta \to  \frac {\pi } {2}$.
It appears that no complete answer is known to this problem. 

We first consider the case $\gamma \ge 0$, for which one can give a partial answer. If $\Ene ( \DI) <0$, then $\Tma ^\theta <\infty $ for all $0\le \theta <\frac {\pi } {2}$ by Theorem~\ref{eBUg1:1}. However, the bound in~\eqref{eBUg1:2} blows up as $\theta \to \frac {\pi } {2}$. The proof of~\eqref{eBUg1:2} is based on Levine's argument for blowup in the nonlinear heat equation~\eqref{NLH}. Since, as observed before,  Levine's argument does not apply to the limiting nonlinear Schr\"o\-din\-ger equation, it is not surprising that the bound in~\eqref{eBUg1:2} becomes inaccurate as $\theta \to \frac {\pi } {2}$. This observation suggests to adapt the proof of blowup for~\eqref{NLS} to equation~\eqref{GL}, in order to obtain a  bound on $\Tma ^\theta $ as $\theta \to \frac {\pi } {2}$. 
This strategy proved to be successful in~\cite[Theorem~1.5]{CazenaveDW} for $\gamma =0$, and it can be extended to the case $\gamma \ge 0$. More precisely, we have the following result.

\begin{thm} \label{eGGL1bu} 
Let $\frac {4} {N} \le \alpha \le 4$, $\gamma \ge 0$ and $N\ge 2$. 
Assume that $N\ge 3$ if $\gamma >0$.
Let $\DI\in H^1 (\R^N ) \cap \Cz$ be radially symmetric and, given any $0\le \theta < \frac {\pi } {2}$,
let $u^\theta$ be the corresponding  solution of~\eqref{GL} defined on the maximal interval $[0, \Tma^\theta )$.
If $\Ene (\DI)<0$, then   $\sup _{ 0\le \theta <\frac {\pi } {2} }\Tma ^\theta <\infty $. 
\end{thm} 

We prove Theorem~\ref{eGGL1bu} by following the strategy of~\cite{CazenaveDW}. We adapt the proof of Theorem~\ref{eNLS3}, and in particular we consider $v(t)= e^{- \gamma t} u(t)$, which satisfies equation~\eqref{GL2b}.
The corresponding identities for the truncated variance are given by Proposition~\ref{eGGLzuu}; and the Caffarelli-Kohn-Nirenberg estimate by Lemma~\ref{eOgau}. 
The terms involving $\cos \theta $ in~\eqref{fVaruqg2}  can be controlled using~\eqref{fPos6}.
Yet there is a major difference between equations~\eqref{NLS} and~\eqref{GL} that must be taken care of. For~\eqref{NLS}, the $L^2$-norm of the solutions is controlled by formula~\eqref{fSL1}.  
The resulting estimate for $v$ is essential when applying Lemma~\ref{eOgau}. For~\eqref{GL}, there is no such  {\it a priori}  estimate. However, one can estimate $v$ on an interval $[0,T]$ where $T$ is proportional to $\Tma^\theta $.
More precisely, we have the following result, similar to~\cite[Lemma~5.2]{CazenaveDW}.

\begin{lem}\label{eGGLq} 
Fix $0\le \theta  < \frac {\pi } {2}$. 
Let $\DI \in \Cz \cap H^1 (\R^N ) $, and consider the 
corresponding solution $v $ of~\eqref{GL2b} defined on the maximal interval $[0, \Tma)$.
 Set
\begin{equation} \label{feGGLq:u} 
\tau = \sup \{ t\in [0,\Tma);\,  \|v(s)\| _{ L^2 }^2 \le K  \| \DI \| _{ L^2 }^2  \text{ for }0\le s\le t  \},
\end{equation} 
where 
\begin{equation} \label{feGGLq:d1} 
K=  \Bigl[ 1 -  \Bigl( \frac {\alpha +4} {2\alpha +4} \Bigr)^{\frac 12} \Bigr] ^{-1} >1
\end{equation} 
 so that $0< \tau \le \Tma$. 
 If $\Ene (\DI) < 0$, then $\Tma \le  \frac {\alpha +4} {\alpha } \tau  $.
\end{lem} 

\begin{proof} 
We use the notation of the proof of Theorem~\ref{eBUg1:1}, and in particular~\eqref{fNot1}. 
The proof is based Levine's argument used in the proof of Theorem~\ref{eNLH1}, which shows that if $\| v \| _{ L^2 }^2$ achieves the value $K  \| \DI \| _{ L^2 }^2$ at a certain time $t $, then $v$ must blow up within a lapse of time which is controlled by $t $. More precisely, let $\tau $ be given by~\eqref{feGGLq:u}. 
If $\tau =\Tma$, there is nothing to prove, so we assume $\tau <\Tma$. It follows that
$\| v (\tau )\| _{ L^2 }^2= K  \| \DI \| _{ L^2 }^2$, so that
\begin{equation} \label{feGGLq:t} 
  \widetilde{f} (t)  \le   \widetilde{f} (\tau )  = K  \widetilde{f} (0) ,\quad 0\le t\le \tau .
\end{equation} 
Since $\Ene ( \DI) < 0$, it follows (see the proof of Theorem~\ref{eBUg1:1}) that $ \widetilde{f} $ is nondecreasing on $[0, \Tma)$; and so, using~\eqref{feGGLq:t},
\begin{equation} \label{feGGLq:q} 
 \widetilde{f} (t) \ge  K  \widetilde{f} (0) ,\quad \tau \le t<\Tma. 
\end{equation} 
We deduce from~\eqref{fPos7} that
\begin{equation*} 
e^{-\alpha \gamma t}  \widetilde{e} (t) \le \Ene (\DI) - \cos \theta \, \int _0 ^t e^{-\alpha \gamma s }  \int  _{ \R^N  }  | v _t (s) |^2 
\end{equation*} 
so that
\begin{equation} \label{feGGLq:qb1} 
 \widetilde{e} (t) \le  - \cos \theta \, \int _0 ^t   \int  _{ \R^N  }  | v _t  |^2.
\end{equation} 
Since $  \frac {d\widetilde{f}} {dt} \ge -2 \cos \theta \,  \widetilde{\jmath} \ge   -2 (\alpha +2)  \cos \theta \,  \widetilde{e}$
by~\eqref{fPos6}, we deduce from~\eqref{feGGLq:qb1} that
\begin{equation} \label{fGGLud}
\frac {d\widetilde{f}} {dt}   \ge   
2(\alpha +2) \cos^2 \theta  \int _0^t \int _{\R^N }  | v_t|^2.
\end{equation}
Set 
\begin{equation} \label{fGGLut}
 \widetilde{M} (t)= \frac {1} {2} \int _0^t   \widetilde{f}(s) \,ds .
\end{equation} 
It follows from~\eqref{fGGLud} and Cauchy-Schwarz's inequality that (see the proof of Theorem~\ref{eNLH1})
\begin{equation}  \label{fGGLuq}
 \widetilde{M}   \widetilde{M} '' \ge  \frac {\alpha +2} {2} \cos ^2 \theta   \Bigl( \int _0^t \Bigl| \int _{\R^N } v_t  \overline{v}  \Bigr| \Bigr)^2 .
\end{equation} 
Since $  \widetilde{\jmath}  \le (\alpha +2)  \widetilde{e}  \le 0$, 
identities~\eqref{fPos6b} and~\eqref{fPos6} yield
\begin{equation*}
\Bigl| \int _{\R^N } v_t  \overline{v}  \Bigr| = -  \widetilde{\jmath} 
=  \frac {1} {2\cos \theta } \frac {d\widetilde{f}} {dt}  
=  \frac {1} {\cos \theta }  \widetilde{M} '' (t) 
\end{equation*} 
so that~\eqref{fGGLuq} becomes
\begin{equation} \label{fGGLuqbd}
 \widetilde{M}   \widetilde{M} '' \ge \frac {\alpha +2} {2} ( \widetilde{M} '(t) - \widetilde{M} '(0))^2
 =  \frac {\alpha +2} {8} ( \widetilde{f} (t) - \widetilde{f} (0))^2.
\end{equation} 
It follows from~\eqref{fGGLuqbd}, \eqref{feGGLq:q} and~\eqref{feGGLq:d1}   that
\begin{equation} \label{feGGLq:c}
 \widetilde{M}   \widetilde{M} '' \ge \frac{\alpha +2} {8}  \Bigl( \frac {K-1} {K}  \Bigr)^2    \widetilde{f}  (t) ^2 = \frac{\alpha +4} {16}  \widetilde{f}  (t) ^2 = \frac{\alpha +4} {4} [  \widetilde{M} '(t)]^2
\end{equation} 
for all $\tau \le t< \Tma$. This means that $( \widetilde{M} ^{-\frac {\alpha } {4} }) '' \le 0$ on $[\tau ,\Tma)$; and so
\begin{equation*} 
 \widetilde{M} (t)^{-\frac {\alpha } {4}} \le  \widetilde{M} (\tau )^{-\frac {\alpha } {4}} + (t-\tau ) ( \widetilde{M} ^{-\frac {\alpha } {4}})'(\tau )
=  \widetilde{M} (\tau )^{-\frac {\alpha } {4}}  \Bigl[ 1- \frac {\alpha } {4}(t-\tau )  \widetilde{M} (\tau )^{-1}  \widetilde{M} '(\tau ) \Bigr]
\end{equation*} 
for $\tau \le t   < \Tma$. Since $ \widetilde{M} (t)^{-\frac {\alpha } {4}}> 0$, we deduce that for every $\tau \le t<\Tma$,
$ \frac {\alpha } {4}(t-\tau )  \widetilde{M} (\tau )^{-1}  \widetilde{M} '(\tau ) \le 1$, i.e.,
\begin{equation} 
(t-\tau )   \widetilde{f} (\tau ) \le \frac {4} {\alpha } \int _0^\tau   \widetilde{f} (s)\,  ds \le  \frac {4} {\alpha } \tau   \widetilde{f} (\tau ) 
\end{equation} 
where we used~\eqref{feGGLq:t} in the last inequality.
Thus $t\le \frac {\alpha +4} {\alpha } \tau $ for all $\tau \le t< \Tma$, which proves the desired inequality.
\end{proof} 

\begin{proof} [Proof of Theorem~$\ref{eGGL1bu}$]
We set $v^\theta (t)= e^{- \gamma t} u^\theta (t)$, thus $v^\theta $ is the solution of~\eqref{GL2b} on $[0, \Tma^\theta )$. We let $K$ be defined by~\eqref{feGGLq:d1} and we set 
\begin{equation} \label{fPRFu1} 
\tau _\theta = \sup \{ t\in [0,\Tma^\theta );\,  \|v^\theta (s)\| _{ L^2 }^2 \le K  \| \DI \| _{ L^2 }^2  \text{ for }0\le s\le t  \} .
\end{equation} 
Therefore, 
\begin{equation}  \label{fPRFd1} 
\sup  _{ 0\le \theta <\frac {\pi } {2} }\sup  _{ 0\le t<\tau _\theta  }  \|v^\theta (t)\| _{ L^2 }^2 \le K \| \DI \| _{ L^2 }^2
\end{equation} 
and, by Lemma~\ref{eGGLq},
\begin{equation}  \label{fPRFt1} 
\Tma ^\theta  \le  \frac {\alpha +4} {\alpha } \tau _\theta  
\end{equation} 
so that we only need a bound on $\tau _\theta $.
We first derive an inequality (\eqref{fVarugz3} below) by calculating a truncated variance.
Let $\GVar \in C^\infty (\R^N ) \cap W^{4, \infty } (\R^N ) $ be  real-valued, nonnegative, and radially symmetric. We set
\begin{align*} 
\zeta _\theta (t) & = \int  _{ \R^N  } \GVar (x)  |v^\theta (t,x)|^2 dx \\
\Hcal ^\theta (t) & =  \int  _{ \R^N  } \Bigl\{  - 2(2- \GVar '' ) |v^\theta _r|^2  + \frac {\alpha}{\alpha+2} e^{\alpha \gamma t} (2N - \Delta  \GVar ) |v^\theta |^{\alpha +2} 
  -\frac {1}{2}  |v^\theta |^2 \Delta^2  \GVar \Bigr\} \\
\Kcal ^\theta (t) & =   \int _{ \R^N  } \Bigl\{ - 2  \GVar   | v^\theta  _r |^2+ \frac {\alpha +4} {\alpha +2} e^{\alpha \gamma  t} \GVar   |v^\theta |^{\alpha +2}  +  |v^\theta |^2 \Delta  \GVar  \Bigr\} 
\end{align*} 
and we observe that
\begin{equation} \label{fVarugz2} 
- \Kcal ^\theta (0) \le C (  \| \GVar \| _{ L^\infty  } + \| \Delta \GVar \| _{ L^\infty  } )  \| \DI \| _{ H^1 }^2 .
\end{equation} 
We apply Proposition~\ref{eGGLzuu} with $f(t)\equiv e^{\alpha \gamma  t}$. It follows from~\eqref{fVarug1} that
\begin{equation} \label{fVarug} 
\frac {1} {2} \zeta _\theta ' (0) \le C (  \| \GVar \| _{ L^\infty  } +  \| \nabla \GVar \| _{ L^\infty  } + \| \Delta \GVar \| _{ L^\infty  } )( 1 + \| \DI \| _{ H^1 }^{\alpha +2} ) 
\end{equation} 
and from~\eqref{fVaruqg2}  that
\begin{equation} \label{fVarugz1} 
\begin{split} 
\frac {1} {2} \zeta_\theta  '' \le &
 -\frac {1}{2} \int _{ \R^N  } |v^\theta |^2  \Delta^2  \GVar - \frac {\alpha e^{\alpha \gamma  t}}{\alpha+2}  \int _{ \R^N  }  |v^\theta |^{\alpha +2} \Delta  \GVar 
 +2 \int _{ \R^N  } \GVar '' |v^\theta _r|^2
\\ & + \cos \theta   \frac {d} {dt} \Kcal^\theta   .
\end{split}  
\end{equation} 
Using the identity
\begin{multline*} 
  -\frac {1}{2} \int _{ \R^N  } |v^\theta |^2  \Delta^2  \GVar - \frac {\alpha e^{\alpha \gamma  t}}{\alpha+2}  \int _{ \R^N  }   |v^\theta |^{\alpha +2} \Delta  \GVar   +2 \int _{ \R^N  } \GVar '' |v^\theta _r|^2 \\  = 2N\alpha  \widetilde{e} (t)  + \Hcal ^\theta (t)  - (N\alpha -4) \int  _{ \R^N  }  |v^\theta _r|^2
\end{multline*} 
where $\widetilde{e} (t) $ is defined by~\eqref{fNot1}, the estimate
$  \widetilde{e} (t) \le e^{\alpha \gamma t} \Ene (\DI)$ by~\eqref{fPos7}, 
and  the assumption $N\alpha \ge 4$,  we deduce from~\eqref{fVarugz1} that
\begin{equation} \label{fVaruqgb} 
\frac {1} {2} \zeta _\theta ''  \le  2N\alpha  e^{\alpha \gamma t} \Ene (\DI) + \Hcal^\theta  (t) 
+ \cos \theta  \frac {d} {dt} \Kcal^\theta  .
\end{equation} 
Integrating twice the above inequality, and since $\zeta _\theta \ge 0$, we obtain
\begin{equation} \label{fVarugz3} 
\begin{split} 
0 \le & \, \tau _\theta  \, \Bigl( \frac {1} {2}  \zeta _\theta ' (0) - \cos \theta  \Kcal ^\theta (0) \Bigr) 
+ 2N\alpha   \Ene (\DI) \int _0^{\tau _\theta } \int _0^s e^{\alpha \gamma \sigma } d\sigma ds \\
& +  \int _0^{\tau _\theta } \int _0^s  \Hcal^\theta  (\sigma )\, d\sigma ds + \cos \theta \int _0^{\tau _\theta } \Kcal^\theta (s) \, ds .
\end{split} 
\end{equation} 
We derive a bound on $\tau _\theta $ from~\eqref{fVarugz3}. 
In order to do so we show that, for large time, the dominating term in the right-hand side is the middle one, which is negative. 

We first obtain an estimate of the last term in~\eqref{fVarugz3}, for which the factor $\cos \theta $ is essential. 
Indeed, it follows from~\eqref{fPos6} that
\begin{equation*} 
\frac {d} {dt} \int _{ \R^N  } |v^\theta |^2 
= 2 \cos \theta   \Bigl( -2\widetilde{e} (t) + \frac {\alpha } {\alpha +2} e^{\alpha \gamma t} \int  _{ \R^N  }  | v^\theta |^{\alpha +2} \Bigr) 
\end{equation*} 
where $\widetilde{e} (t) $ is defined by~\eqref{fNot1}. Since $\widetilde{e} (t) \le 0$ by~\eqref{fPos7}, we deduce by integrating on $(0,\tau _\theta )$ and applying~\eqref{fPRFd1} that
\begin{equation*}
\frac {2\alpha } {\alpha +2} \cos \theta  \int _0^t e^{\alpha \gamma s} \int  _{ \R^N  }  |v^\theta |^{\alpha +2}\le     (K-1)  \|  \DI \| _{ L^2 }^2 
\end{equation*} 
for all $0\le t\le \tau _\theta $. It follows that
\begin{equation}  \label{fPRFh1}
 \cos \theta \int _0^{\tau _\theta } \Kcal^\theta (s) \, ds \le  C (  \| \GVar \| _{ L^\infty  } + \| \Delta \GVar \| _{ L^\infty  } )  \|  \DI \| _{ L^2 }^2
\end{equation} 
where we used again~\eqref{fPRFd1} to estimate the factor of $\Delta \GVar$.

We conclude by estimating the term involving $\Hcal^\theta $ in~\eqref{fVarugz3} with Lemma~\ref{eOgau}.
We first consider the case $\gamma =0$. We apply Lemma~\ref{eOgau} with $A=  \| \DI \| _{ L^2 }$,  $\Avoir = \frac {\alpha } {\alpha +2}$ and $\varepsilon >0$ chosen sufficiently small so that $ \Const    \Avoir  \varepsilon ^ {2(N-1) } <1$ and $\kappa ( \mu , \varepsilon  ) \le -N\alpha  \Ene ( \DI)$. With $\GVar = \GVar_\varepsilon $ given by Lemma~\ref{eOgau}, it follows from~\eqref{eOgau:u} that $\Hcal^\theta (t) \le -N\alpha  \Ene ( \DI)$ for all $0\le t<\tau _\theta $. 
Therefore, we deduce from~\eqref{fVarugz3} and~\eqref{fPRFh1}
\begin{equation*} 
0 \le  \, \tau _\theta  \, \Bigl( \frac {1} {2}  \zeta _\theta ' (0) - \cos \theta  \Kcal ^\theta (0) \Bigr) 
+ N\alpha  \frac {(\tau _\theta )^2} {2} \Ene (\DI) .
\end{equation*} 
Since $ \Ene (\DI) <0$, we conclude that $\sup _{ 0\le \theta <\frac {\pi } {2} }\tau _\theta <\infty $.

We next consider the case $\gamma >0$ (and so $N\ge 3$). We apply Lemma~\ref{eOgau}, this time with  with $A=  \| \DI \| _{ L^2 }$ and
\begin{equation} \label{fEC1} 
\Avoir = \Avoir _\theta = e^{\alpha \gamma \tau _\theta } .
\end{equation} 
The additional difficulty with respect to the case $\gamma =0$ is that $ \Avoir _\theta $ may, in principle, be large.
We fix $\lambda $ satisfying~\eqref{fOT4} (we use the assumption $N\ge 3$) and we set 
\begin{equation}  \label{fEC2} 
\varepsilon _\theta = a  \Avoir _\theta ^{-\lambda } \le a . 
\end{equation} 
Here, $0<a \le 1$ is 
chosen sufficiently small so that $ \Const  a ^ {2(N-1) } <1$, where  $\Const  $ is the constant in Lemma~\ref{eOgau}. Since $\Avoir_\theta  \ge 1$ and $1-2(N-1)\lambda <0$ by~\eqref{fOT4}, it follows in particular that $ \Const    \Avoir _\theta  \varepsilon _\theta  ^ {2(N-1) }=  \Const   \Avoir _\theta ^{1- 2(N-1) \lambda } a   ^ {2(N-1) }    \le  \Const  a   ^ {2(N-1) }   <1$. Moreover, we deduce from~\eqref{fOT4} that $\kappa $ defined by~\eqref{eOgau:u1} satisfies $\kappa (\Avoir_\theta  ,\varepsilon _\theta ) \le C \Avoir _\theta  ^{1- \delta }$, where $C$ is independent of $\theta $, and $\delta >0$ is given by~\eqref{fDdel}. 
We now let $\GVar = \GVar_{\varepsilon _\theta }$ where $\GVar_{\varepsilon  }$ is given by Lemma~\ref{eOgau} 
for this choice of  $\varepsilon $, and it follows from~\eqref{eOgau:1zb} and~\eqref{eOgau:u} that
\begin{equation} \label{fPRFq11} 
\Hcal ^\theta (t)  \le  C e^{(1-\delta )\alpha \gamma \tau _\theta }
\end{equation} 
and from~\eqref{eOgau:1b}, \eqref{fEC2} and~\eqref{fEC1}  that
\begin{equation}  \label{fVarugz4} 
\| \GVar _{\varepsilon _\theta } \| _{ L^\infty  } +  \| \nabla \GVar _{\varepsilon _\theta } \| _{ L^\infty  } + \| \Delta \GVar _{\varepsilon _\theta } \| _{ L^\infty  } \le C  e^{2\lambda \alpha \gamma \tau _\theta }.
\end{equation} 
Finally, we estimate the first term in the right-hand side of~\eqref{fVarugz3} and we deduce from~\eqref{fVarug}, \eqref{fVarugz2}, \eqref{fVarugz4} and~\eqref{fEC1} that
\begin{equation} \label{fVarugz5} 
\frac {1} {2}  \zeta _\theta ' (0) - \cos \theta  \Kcal ^\theta (0) \le C e^{2\lambda \alpha \gamma \tau _\theta }.
\end{equation} 
Estimates~\eqref{fVarugz3}, \eqref{fVarugz5}, \eqref{fPRFq11}, \eqref{fPRFh1}, \eqref{fVarugz4} and~\eqref{fEC1} now yield
\begin{equation} \label{fVarugz6} 
0 \le  C (1+ \tau _\theta ) e^{2\lambda \alpha \gamma \tau _\theta }
+   \frac {2N} {\alpha \gamma ^2}  \Ene (\DI ) (e^{\alpha \gamma \tau _\theta }- 1 - \alpha \gamma \tau _\theta ) 
 +  C \tau _\theta ^2 e^{(1-\delta )\alpha \gamma \tau _\theta }  .
\end{equation} 
Since $\Ene (\DI )<0$, and $\max\{ 2\lambda , 1-\delta  \}<1$,
the right-hand side of the above inequality is negative if $\tau _\theta $ is large. Thus 
$\sup _{ 0\le \theta <\frac {\pi } {2} }\tau _\theta <\infty $, which completes the proof. 
\end{proof}  

\begin{rem} \label{eRem2} 
Under the assumptions of Theorem~\ref{eGGL1bu}, we know that $\Tma^\theta $ remains bounded. On the other hand,  we do not know if  $\Tma^\theta $ has a limit as $\theta \to \frac {\pi } {2}$, and if it does, if this limit is the blowup time of the solution of the limiting Schr\"o\-din\-ger equation. 
\end{rem} 

We end this section by considering the case $\gamma <0$. 
The condition for blowup in Theorem~\ref{eBUg1} in this case is $\Ene  _{ \frac {\gamma } {\cos \theta } } ( \DI )<0$. Given $\DI \in \Cz \cap H^1 (\R^N ) $, $\DI \not = 0$, it is clear that $\Ene  _{ \frac {\gamma } {\cos \theta } } ( \DI ) > 0$ for all $\theta $ sufficiently close to $\frac {\pi } {2}$, and we do not know if there exists an initial value $\DI$ such that the corresponding solution of~\eqref{GL} blows up in finite time for all $\theta $ close to $\frac {\pi } {2}$. (See Open Problem~\ref{OPqneg}.)  

Another point of view concerning the case $\gamma <0$ is to apply the transformation~\eqref{fSC3}--\eqref{fSC2} and study the resulting equation~\eqref{GL2}. Let $\DIb \in \Cz \cap H^1 (\R^N ) $ and, given $0\le  \theta < \frac {\pi } {2}$,  let $v^\theta $ the corresponding solution of~\eqref{GL2} defined on the  maximal interval $[0, \Sma^\theta )$.
If $ \Ene _{ -1 } (\DIb )<0$, then it follows from~\eqref{fES1b3} that $\Sma^\theta <\infty $ for all $0\le  \theta < \frac {\pi } {2}$. Therefore, it makes sense to study the behavior of $\Sma^\theta $ as $\theta \to \frac {\pi } {2}$, and we have the following result.

\begin{thm}\label{eGGLhbu} 
Suppose $N\ge 2$, $\frac {4} {N} \le \alpha \le 4$,
and fix a radially symmetric initial value $\DIb \in H^1 (\R^N ) \cap \Cz$.
Given any $0\le \theta < \frac {\pi } {2}$,
let $v^\theta $ be the corresponding solution of~\eqref{GL2} defined on the maximal interval $[0, \Sma^\theta )$.
If $ \Ene _{ -1 } (\DIb )<0$, then $\sup _{ 0\le \theta <\frac {\pi } {2} }\Tma ^\theta <\infty $. 
\end{thm} 

The proof of Theorem~\ref{eGGLhbu} is very similar to the proof of~\cite[proof of Theorem~1.5]{CazenaveDW},  with minor modifications only. 
More precisely, it is not difficult to adapt the proof of Lemma~\ref{eGGLq} to show that if
\begin{equation*}
\tau _\theta = \sup \{ t\in [0,\Sma^\theta );\,  \|v^\theta (s)\| _{ L^2 }^2 \le K  \| \DIb  \| _{ L^2 }^2  \text{ for }0\le s\le t  \},
\end{equation*} 
where $K$ is defined by~\eqref{feGGLq:d1},  then $\Sma ^\theta \le  \frac {\alpha +4} {\alpha } \tau _\theta  $.
Moreover, given a real-valued, radially symmetric function $\GVar \in C^\infty (\R^N ) \cap W^{4, \infty } (\R^N ) $, and setting
\begin{equation*} 
\zeta _\theta (t)= \int  _{ \R^N  } \GVar (x)  |v^\theta (t,x)|^2 dx
\end{equation*} 
it is not difficult to deduce from Proposition~\ref{eGGLzuu} the variance identities
\begin{equation*} 
\begin{split} 
\frac {1} {2} \zeta _\theta '(t)   =  & \cos \theta  \int  _{ \R^N  }   \Bigl\{ - \GVar   | v^\theta _r |^2 +  \GVar   |v^\theta |^{\alpha +2}  - \GVar  |v^\theta |^2 + \frac 12   |v^\theta |^2  \Delta  \GVar \Bigr\}
\\ & +\sin \theta  \Im \int _{ \R^N  }  \overline{v^\theta } (\nabla v^\theta  \cdot  \nabla  \GVar )     
\end{split} 
\end{equation*}
and
\begin{equation*} 
\begin{split} 
\frac {1} {2} &\zeta _\theta '' (t) \le  2N\alpha  \Ene _{ -1 }(v^\theta ) \\ & + \int  _{ \R^N  }  \Bigl\{ -2(2- \GVar '' ) |v^\theta _r|^2  + \frac {\alpha}{\alpha+2}    (2N - \Delta  \GVar ) |v^\theta |^{\alpha +2} 
  -\frac {1}{2}  (2N\alpha + \Delta^2  \GVar ) |v^\theta |^2 \Bigr\}
\\  &+ \cos \theta   \frac {d} {dt}   \int _{ \R^N  } \Bigl\{ - 2  \GVar   | v^\theta _r|^2+ \frac {\alpha +4} {\alpha +2}   \GVar   |v^\theta |^{\alpha +2} + ( \Delta  \GVar  -2 \GVar ) |v^\theta |^2 \Bigr\}  .
\end{split}  
\end{equation*} 
One can then conclude exactly as in the case $\gamma =0$ of the proof of Theorem~\ref{eGGL1bu}.

\section{Some open problems} \label{OPPB} 

\begin{opb} 
Suppose $\frac {4} {N} < \alpha < \frac {4} {N-2}$. Let  $\DI\in H^1 (\R^N ) $, $\DI \not = 0$ and $u$ the corresponding solution of~\eqref{NLS}. It follows from~\cite[Theorem~1]{OhtaT} that, if $\gamma  $ is sufficiently negative, then $u$ is global. 
Does $u$ blows up in finite time for $\gamma >0$ sufficiently large (this is true if $\Ene (\DI) <0$ and $\DI \in L^2 (\R^N ,  |x|^2 dx) $, by Theorem~$\ref{eNLS2}$), or does there exist $\DI\not = 0$ such that $u$ is global for all $\gamma >0$? 
\end{opb} 

\begin{opb} 
Let $0\le \theta <\frac {\pi } {2}$, $\DI \in \Cz \cap H^1 (\R^N ) $, $\DI\not = 0$, and let $u$ be the corresponding solution of~\eqref{GL}. 
If $\gamma $ is sufficiently negative, then $u$ is global, by Remark~$\ref{eLocGL3}$.
Does $u$ blows up in finite time for all sufficiently large $\gamma >0$ (this is true if $\Ene (\DI) <0$, by Theorem~$\ref{eBUg1}$), or does there exist $\DI\not = 0$ such that $u$ is global for all $\gamma >0$? (The question is open even for the nonlinear heat equation~\eqref{NLH}.)
\end{opb} 

\begin{opb} \label{OPqneg} 
Let $\gamma <0$. Does there exist an initial value $\DI \in \Cz \cap H^1 (\R^N ) $, $\DI \not = 0$ such that the corresponding solution of~\eqref{GL} blows up in finite time for all $\theta $ close to $\frac {\pi } {2}$? 
One possible strategy for constructing such initial values when $\frac {4} {N}< \alpha <\frac {4} {N-2}$, would be to adapt the proof of Theorem~$\ref{eNLS2b}$ to equation~\eqref{GL}.
\end{opb} 

\begin{opb} 
Suppose $\gamma \ge 0$ and $0\le \theta <\frac {\pi } {2}$. The sufficient condition for blowup in Theorem~$\ref{eBUg1}$ is $\Ene (\DI) <0$. 
Does there exists a constant $\kappa > \frac {2} {\alpha +2}$ such that the (weaker) condition $\int  _{ \R^N  } |\nabla \DI |^2 - \kappa \int  _{ \R^N  }  | \DI |^{\alpha +2} <0$ implies finite-time blowup?
Note that for the equation with $\gamma =0$ set on a bounded domain with Dirichlet boundary conditions, $\kappa =1$ is not admissible. Indeed, there exist initial values for which $\Neh ( \DI )<0$ and $\Tma = \infty $. (See~\cite{DicksteinMSW}).  
\end{opb} 

\begin{opb} 
Theorems~$\ref{eGlob1}$ and~$\ref{eGGLhbu}$ require that  $\alpha \le 4$ and the solution is radially symmetric. Are these assuptions necessary?  
Note that they are necessary in Lemma~{\rm\ref{eOgau}} (see Section~$6$ in~\cite{CazenaveDW}) which is an essential tool in our proof.
Could these assumption be replaced by stronger decay conditions on the initial value, such as $ \DI \in L^2 (\R^N ,  |x|^2 dx) $? 
In particular, one could think of adapting the proof of Theorem~$\ref{eNLS2}$ (instead of the proof of Theorem~$\ref{eNLS3}$), but this does not seem to be simple, see Section~$7$ in~\cite{CazenaveDW}.

\end{opb} 

\appendix

\section{A truncated variance identity} \label{TVI} 

We prove the following result, which is a slightly more general form of~\cite[Lemma~5.1]{CazenaveDW}.

\begin{prop}\label{eGGLzuu} 
Fix $\alpha >0$, $0\le \theta  <\frac {\pi } {2}$, and a real-valued function $\GVar \in C^\infty (\R^N ) \cap W^{4, \infty } (\R^N ) $. 
Let $\DI \in \Cz \cap H^1 (\R^N ) $,   $f  \in C^1(\R, \R)$,  and consider 
the  corresponding  solution
 $v$  of
\begin{equation} \label{GL2b1}
\begin{cases}
 v_t= e^{i\theta } [\Delta v+ f(t) |v|^\alpha v ] \\
v(0)= \DI
\end{cases}
\end{equation} 
 defined on the maximal interval $[0, \Tma)$.
If $\zeta $ is defined by
\begin{equation*} 
\zeta (t)= \int  _{ \R^N  } \Psi (x)  |v(t,x)|^2 dx
\end{equation*} 
then  $\zeta \in C^2([0, \Tma))$, 
\begin{equation} \label{fVarug1} 
\begin{split} 
\frac {1} {2} \zeta '(t)   =  & \cos \theta    \int  _{ \R^N  }  \Bigl\{ - \GVar   |\nabla v|^2 + f(t) \GVar   |v|^{\alpha +2} + \frac 12  |v|^2  \Delta  \GVar \Bigr\}
\\ & +\sin \theta  \Im \int _{ \R^N  }  \overline{v} ( \nabla v \cdot  \nabla  \GVar) 
\end{split} 
\end{equation}
and
\begin{equation} \label{fVaruqg1} 
\begin{split} 
\frac {1} {2} \zeta '' (t) =&
  \int _{ \R^N  }  \Bigl\{ -\frac {1}{2} |v|^2 \Delta^2  \GVar  - \frac {\alpha f(t)}{\alpha+2}   |v|^{\alpha +2}  \Delta  \GVar  +2\Re  \langle H(\GVar ) \nabla \overline{v},\nabla v\rangle   \Bigr\}
\\  + \cos \theta &  \frac {d} {dt}   \int _{ \R^N  } \Bigl\{ - 2  \GVar   |\nabla v|^2+ \frac {\alpha +4} {\alpha +2}  f(t) \GVar   |v|^{\alpha +2} +   |v|^2 \Delta  \GVar \Bigr\}  
\\   & -2 \cos^2 \theta \int _{ \R^N  }  \GVar   |v_t|^2  - \frac {2f'(t)} {\alpha +2}   \cos \theta  \int  _{ \R^N  }  \GVar   |v|^{\alpha +2}  
\end{split}  
\end{equation} 
for all $0\le t<\Tma$, where $H( \GVar )$ is the Hessian matrix $(\partial ^2 _{ ij }\GVar ) _{ i,j }$.
In particular, if both $\GVar$ and $\DI$ (hence, $v$) are radially symmetric,  then
\begin{equation} \label{fVaruqg2} 
\begin{split} 
\frac {1} {2} \zeta '' (t) =&
\int _{ \R^N  }   \Bigl\{ -\frac {1}{2}  |v|^2 \Delta^2  \GVar - \frac {\alpha f(t)}{\alpha+2}  |v|^{\alpha +2} \Delta  \GVar   +2 \GVar '' |v_r|^2 \Bigr\}
\\  + \cos \theta &  \frac {d} {dt}   \int _{ \R^N  } \Bigl\{ - 2  \GVar   | v _r |^2+ \frac {\alpha +4} {\alpha +2} f(t) \GVar   |v|^{\alpha +2} +  |v|^2  \Delta  \GVar \Bigr\}  
\\   & -2 \cos^2 \theta \int _{ \R^N  }  \GVar   |v_t|^2  -  \frac {2f'(t)} {\alpha +2}   \cos \theta  \int  _{ \R^N  }  \GVar   |v|^{\alpha +2} 
\end{split}  
\end{equation} 
for all $0\le t<\Tma$.
\end{prop} 

\begin{proof} 
Multiplying the equation~\eqref{GL2b1}  by $\GVar (x) \overline{v} $, taking the real part, and using the identity
\begin{equation*} 
2 \Re [ \overline{v} ( \nabla v \cdot \nabla \GVar  )] = \nabla \cdot (|v|^2 \nabla  \GVar  ) -  |v|^2 \Delta  \GVar  
\end{equation*} 
we obtain~\eqref{fVarug1}. 
Next,  the identity
\begin{equation*} 
 \overline{v} ( \nabla v_t \cdot \nabla  \GVar ) = \nabla \cdot (  v_t  \overline{v} \nabla  \GVar  ) - (\nabla  \GVar  \cdot \nabla  \overline{v}) v_t -  \overline{v} v_t \Delta  \GVar   
\end{equation*} 
and integration by parts yield
\begin{equation*} 
\frac {d} {dt}  \Bigl(  \sin \theta  \Im \int _{ \R^N  }  \overline{v} ( \nabla v \cdot \nabla  \GVar ) \Bigr) =-  \sin \theta 
\Im \int _{ \R^N  } [  \overline{v} \Delta  \GVar     + 2   \nabla  \GVar  \cdot \nabla   \overline{v}   ] v_t. 
\end{equation*} 
We rewrite this last identity in the form
\begin{equation} \label{fVarqg} 
\begin{split} 
\frac {d} {dt}  \Bigl(  \sin \theta  \Im \int _{ \R^N  }  \overline{v} ( \nabla v \cdot \nabla  \GVar ) \Bigr) =  &\cos  \theta \,  \Re \int _{ \R^N  } [  \overline{v}
\Delta  \GVar  +2   \nabla  \GVar  \cdot \nabla  \overline{v} ] v_t \\  & -  
\Re \int _{ \R^N  } [  \overline{v}  \Delta  \GVar  + 2   \nabla  \GVar  \cdot \nabla   \overline{v}   ] e^{-i \theta }v_t . 
\end{split} 
\end{equation} 
Using~\eqref{GL2b1} and the identities
\begin{align*} 
\Re [ (\nabla  \GVar \cdot \nabla  \overline{v})  |v|^\alpha v ] &= \frac {1} {\alpha +2} \nabla \cdot (  |v|^{\alpha +2} \nabla  \GVar ) -
\frac {1} {\alpha +2} |v|^{\alpha +2} \Delta  \GVar  \\
\Re [  \Delta  v (  \overline{v} \Delta \GVar +2    \nabla  \overline{v} \cdot \nabla  \GVar ) ] &=
\Re \nabla \cdot   \Bigl[ \nabla v ( \overline{v} \Delta \GVar + 2 \nabla  \overline{v} \cdot \nabla \GVar ) -  |\nabla v|^2 \nabla \GVar \\ & - \frac {1} {2}  |v|^2 \nabla (\Delta \GVar) \Bigr]
 -2 \Re \langle H( \GVar ) \nabla \overline{v},\nabla v \rangle 
+ \frac {1} {2}  | v|^2 \Delta ^2 \GVar 
\end{align*} 
 we see that
\begin{equation} \label{fVarcg} 
\begin{split} 
- &
\Re  \int _{ \R^N  } [   \overline{v} \Delta  \GVar     + 2   \nabla  \GVar  \cdot \nabla   \overline{v}   ] e^{-i \theta } v_t  \\ &
= -\Re \int _{ \R^N  }  [   \overline{v} \Delta  \GVar     + 2   \nabla  \GVar  \cdot \nabla   \overline{v}   ] (\Delta v+ f(t) |v|^\alpha v )  \\  =& -\frac {1}{2} \int _{ \R^N  }  |v|^2 \Delta^2  \GVar  - \frac {\alpha f(t)}{\alpha+2}   \int _{ \R^N  }  |v|^{\alpha +2}  \Delta  \GVar 
  +2\Re \int _{ \R^N  } \langle H( \GVar ) \nabla \overline{v},\nabla v\rangle .
\end{split} 
\end{equation}  
Next, we observe that
\begin{equation} \label{fLS1} 
 \Re \int _{ \R^N  } [  \overline{v}
\Delta  \GVar  +2   \nabla  \GVar  \cdot \nabla  \overline{v} ] v_t
= \frac {1} {2} \frac {d} {dt} \int  _{ \R^N  }  |v|^2 \Delta \GVar + 2 \Re \int  _{ \R^N  } (\nabla  \overline{v} \cdot \nabla \GVar ) v_t .
\end{equation} 
On the other hand,
\begin{equation*}
\begin{split} 
&  \frac {f' } {\alpha +2}  \int  _{ \R^N  } \GVar   |v|^{\alpha +2} +
\frac {d} {dt}\int _{ \R^N  }  \GVar   \Bigl( \frac { |\nabla v|^2} {2} - \frac {f(t)} {\alpha +2}  |v|^{\alpha +2} \Bigr)   \\ & =
\Re \int  _{ \R^N  } \GVar  (\nabla  \overline{v} \cdot \nabla v_t - f |v|^\alpha  \overline{v} v_t  )  
= - \Re \int _{ \R^N  } [   \GVar  (\Delta  \overline{v} + f |v|^\alpha  \overline{v} ) v_t   +  (\nabla  \GVar \cdot \nabla  \overline{v}) v_t ]   \\ & = -\cos \theta \int  _{ \R^N  } \GVar   |v_t|^2 -\Re \int _{ \R^N  } (\nabla  \GVar  \cdot \nabla  \overline{v})v_t   
\end{split} 
\end{equation*} 
so that
\begin{multline}  \label{fVaruug} 
2 \Re \int _{ \R^N  } (\nabla  \GVar  \cdot \nabla  \overline{v})v_t  = -2 \cos \theta \int  _{ \R^N  } \GVar   |v_t|^2 -  \frac {2f' } {\alpha +2}  \int  _{ \R^N  } \GVar   |v|^{\alpha +2}  \\
- \frac {d} {dt}\int _{ \R^N  }  \GVar   \Bigl(  { |\nabla v|^2}  - \frac {2f(t)} {\alpha +2}  |v|^{\alpha +2}   \Bigr) .
\end{multline} 
Applying~\eqref{fVarqg}, \eqref{fVarcg}, \eqref{fLS1}   and~\eqref{fVaruug}, we deduce that
\begin{equation} \label{fVarutg} 
\begin{split} 
\frac {d} {dt} & \Bigl(    \sin \theta  \Im \int _{ \R^N  }  \overline{v} (\nabla  \GVar \cdot  \nabla v )  \Bigr)   \\ &=  \int _{ \R^N  }  \Bigl\{  -\frac {1}{2} |v|^2 \Delta^2  \GVar  - \frac {\alpha f(t)}{\alpha+2}    |v|^{\alpha +2}  \Delta  \GVar 
  +2\Re \langle H( \GVar ) \nabla \overline{v},\nabla v\rangle  \Bigr\} \\  & 
+ \cos \theta  \frac {d} {dt}   \int   _{ \R^N  } \Bigl( - \GVar  |\nabla v|^2 + \frac { 2 f(t)} {\alpha +2}  \GVar   |v|^{\alpha +2}  + \frac {1}{2}  |v|^2 \Delta  \GVar \Bigr)  
\\  &  -2 \cos^2 \theta \int _{ \R^N  }  \GVar   |v_t|^2  - \frac {2f'(t)} {\alpha +2}   \cos \theta  \int  _{ \R^N  } \GVar   |v|^{\alpha +2} .
\end{split} 
\end{equation} 
Taking now the time derivative of~\eqref{fVarug1} and applying~\eqref{fVarutg}, we obtain~\eqref{fVaruqg1}. Finally, if both $\GVar$ and $\DI$ (hence, $v$) are radially symmetric, then $ \Re \langle H( \GVar ) \nabla \overline{v},\nabla v\rangle =  \GVar '' |v_r|^2 $,
 so that~\eqref{fVaruqg2} follows from~\eqref{fVaruqg1}.
\end{proof} 

\section{A Caffarelli-Kohn-Nirenberg inequality} \label{LOT} 

We use the following form of Caffarelli-Kohn-Nirenberg inequality~\cite{CaffarelliKN}. It extends an inequality which was established in~\cite{OgawaTu} and generalized in~\cite[Lemma~5.3]{CazenaveDW}.

\begin{lem} \label{eOgau} 
Suppose $N\ge 2$ and $\alpha \le 4$ and let $  \Granda> 0$.
There exist a constant $\Const $ and a family $ ( \GVar _\varepsilon  )  _{ \varepsilon >0 } \subset  C^\infty (\R^N ) \cap W^{4,\infty } (\R^N ) $ of radially symmetric functions such that $\GVar _\varepsilon (x)>0$ for $x\not = 0$,
\begin{gather} 
\sup  _{ \varepsilon >0 }\, [ \varepsilon ^2 \| \GVar_\varepsilon  \| _{ L^\infty  } + \varepsilon   \| \partial _r\GVar_\varepsilon \| _{ L^\infty  }
+  \| \Delta \GVar _\varepsilon  \| _{ L^\infty  } 
+ \varepsilon ^{-2}  \| \Delta ^2 \GVar _\varepsilon  \| _{ L^\infty  } ] < \infty \label{eOgau:1b} \\
2N- \Delta \GVar_\varepsilon \ge 0 \label{eOgau:1zb} 
\end{gather} 
 and
\begin{equation} \label{eOgau:u} 
\begin{split} 
 -2  \int _{ \R^N  } (2 - \GVar_\varepsilon  '' ) |u_r|^2 
 &
  +\Avoir \int _{ \R^N  }  (2N-  \Delta  \GVar_\varepsilon )  |u|^{\alpha +2}
 \\ & 
 -\frac {1}{2} \int _{ \R^N  }  |u|^2 \Delta^2  \GVar _\varepsilon  \le 
\kappa ( \mu , \varepsilon  ) 
\end{split} 
\end{equation} 
for all radially symmetric $u\in H^1 (\R^N ) $ such that $ \|u\| _{ L^2 }\le  \Granda $ and all 
$\mu , \varepsilon >0$ such that $ \Const    \Avoir  \varepsilon ^ {2(N-1) } <1$,
where 
\begin{equation} \label{eOgau:u1} 
\kappa ( \mu , \varepsilon  ) = 
\begin{cases} 
 \Const \Avoir  \Bigl(  \varepsilon ^{\frac {N\alpha } {2}}  
 +  [ \Const    \Avoir  \varepsilon ^ {2(N-1) }  ] ^{\frac {\alpha } {4-\alpha }} \Bigr) +\Const \varepsilon ^2  & \text{if } 0<\alpha <4 \\
  \Const \Avoir    \varepsilon ^{\frac {N\alpha } {2}}   +\Const \varepsilon ^2  & \text{if } \alpha =4. 
\end{cases} 
\end{equation} 
\end{lem} 

\begin{proof} 
We follow the method of~\cite{OgawaTu}, and we construct a family $(\GVar _ \varepsilon ) _{ \varepsilon >0 }$ such that, given $  \Granda $, the estimate~\eqref{eOgau:u} holds with $\GVar = \GVar _\varepsilon $ provided $\varepsilon >0$ is sufficiently small.
Fix  a function $h\in C^\infty  ([0, \infty ))$ such that
\begin{equation*} 
h\ge 0,\quad \Supp h\subset [1,2 ], \quad \int _0^\infty h(s)\,ds =1
\end{equation*} 
 and let
\begin{equation*} 
\zeta (t)= t- \int _0^t (t-s) h(s)\,ds = t- \int _0^t \int _0^s h(\sigma )\,d\sigma ds
\end{equation*} 
for $t\ge 0$. It follows that $\zeta \in C^\infty ([0,\infty )) \cap W^{4,\infty }((0,\infty ))$, 
$\zeta '\ge 0$, $ \zeta ''\le 0$, 
$\zeta (t)= t$ for $t\le 1$, and $\zeta (t)=M$ for $t\ge 2$ with $M= \int _0^2 sh(s)\,ds$.
Set
\begin{equation*} 
\Phi (x)= \zeta (  |x|^2) .
\end{equation*} 
It follows in particular that $\Phi \in C^\infty (\R^N ) \cap W^{4, \infty } (\R^N ) $.
Given any $\varepsilon >0$, set
\begin{equation*} 
\GVar _\varepsilon  (x)= \varepsilon ^{-2 } \Phi (\varepsilon x),
\end{equation*} 
so that
\begin{equation} \label{fYTMuz} 
  \| D^\beta  \GVar _\varepsilon  \| _{ L^\infty  }= \varepsilon ^{ |\beta |-2}  \| D^\beta  \Phi \| _{ L^\infty  }
\end{equation} 
where $\beta $ is any multi-index such that $0\le  |\beta |\le 4$.
Next, set
\begin{equation} \label{fLMmd} 
\xi (t)= \sqrt{2( 1-\zeta '(t)) -4t\zeta ''(t)}= \sqrt{2\int _0^t h(s)\,ds +4 th(t)}.
\end{equation} 
It is not difficult to check that $\xi \in C^1([0,\infty )) \cap W^{1, \infty }(0,\infty )$. 
Let
\begin{equation*} 
\gamma (r)= \xi (r^2)
\end{equation*} 
and, given $\varepsilon >0$, let
\begin{equation*} 
\gamma _\varepsilon (r)= \gamma (\varepsilon r) .
\end{equation*} 
It easily follows that $\gamma _\varepsilon $ is supported in $[\varepsilon ^{-1}, \infty )$, so that 
\begin{align} 
\| r^{-(N-1)} \gamma _\varepsilon ' \| _{ L^\infty  } \le \varepsilon ^{N-1} \|  \gamma _\varepsilon ' \| _{ L^\infty  }=  \varepsilon ^N  \|\gamma '\| _{ L^\infty  } \label{fLMmc} \\
  \| r^{-(N-1)} \gamma _\varepsilon u_r\| _{ L^2 }  \le \varepsilon ^{N-1} 
    \|  \gamma _\varepsilon u_r\| _{ L^2 } . \label{fLMms} 
\end{align} 
Set
\begin{equation} \label{fLMmuz} 
\begin{split} 
I_\varepsilon (u) & =
 -2  \int _{ \R^N  } (2 - \GVar _\varepsilon '' ) |u_r|^2 
 + \Avoir \int _{ \R^N  }  (2N-  \Delta  \GVar_\varepsilon )  |u|^{\alpha +2} \\ &
  -\frac {1}{2} \int _{ \R^N  }  |u|^2 \Delta^2  \GVar _\varepsilon . 
\end{split} 
\end{equation} 
Elementary but long calculations using in particular~\eqref{fLMmd} show that
\begin{gather} 
2- \GVar _\varepsilon ''(x)= \gamma _\varepsilon ( |x|)^2 \label{fDefHeps} \\
2N- \Delta \GVar _\varepsilon (x) 
= N  [ \gamma _\varepsilon ( |x|)]^2 + 4(N-1) (\varepsilon  |x|)^2 \zeta ''(\varepsilon ^2 |x|^2) \le 
N  [ \gamma _\varepsilon ( |x|)]^2.  \label{fGVlap} 
\end{gather} 
We deduce from~\eqref{fLMmuz}, \eqref{fDefHeps},  \eqref{fGVlap}, and~\eqref{fYTMuz} that
\begin{equation} \label{fLMmud} 
I_\varepsilon (u) \le -2\int _{ \R^N  } \gamma _\varepsilon ^2  |u_r|^2 + N \Avoir \int _{ \R^N  } \gamma _\varepsilon ^2  |u|^{\alpha +2} + \frac {\varepsilon ^2} {2}  \|\Delta^2 \Phi \| _{ L^\infty  }   \|u\| _{ L^2 } ^2. 
\end{equation} 
We next claim that
\begin{equation} \label{fCLu} 
\| \gamma _\varepsilon ^{\frac {1} {2}} u \| _{ L^\infty  }^2 
\le \varepsilon ^N  \|  \gamma  ' \| _{ L^\infty  }
   \|u\| _{ L^2 }^2  + 2 \varepsilon ^{N-1}    \|u\| _{ L^2 }  
    \|  \gamma _\varepsilon  u_r\| _{ L^2 } .
\end{equation} 
Indeed,
\begin{equation} \label{fLMmh} 
\begin{split} 
\gamma _\varepsilon (r)  |u(r)|^2 & = - \int _r^\infty \frac {d} {ds}[\gamma _\varepsilon (s)  |u(s)|^2] 
 \le \int _0^\infty   |\gamma _\varepsilon '|\,  |u|^2 + 2 \int _0^\infty \gamma _\varepsilon  |u|\,  |u_r| \\
& \le   \| r^{-(N-1)} \gamma _\varepsilon ' \| _{ L^\infty  }  \|u\| _{ L^2 }^2 + 2  \|u\| _{ L^2 }  \| r^{-(N-1)} \gamma _\varepsilon u_r\| _{ L^2 }.
\end{split} 
\end{equation} 
(The above calculation is valid for a smooth function $u$ and is easily justified for a general $u$ by density.)
The estimate~\eqref{fCLu}  follows from~\eqref{fLMmh}, \eqref{fLMmc}, and~\eqref{fLMms}.   

In what follows, $\Const $ denotes a constant that may depend on $ N, \gamma , \Phi$ and $ \Granda $ and change from line to line, but is independent of $0<\alpha \le 4$ and $\varepsilon >0$.
We  assume $  \|u\| _{ L^2 } \le  \Granda $, and we observe that 
\begin{equation} \label{fLMmp} 
\begin{split} 
\int _{ \R^N  } \gamma _\varepsilon ^2 |u|^{\alpha +2} & = \int _{ \R^N  }
\gamma _\varepsilon ^{\frac {4- \alpha } {2}}
[ \gamma _\varepsilon^{\frac {1} {2}}  |u|  ]^\alpha 
 |u|^2\le   \| \gamma  \| _{ L^\infty  }^{\frac {4- \alpha } {2}}  \| \gamma _\varepsilon ^{\frac {1} {2}} u \| _{ L^\infty  }^\alpha   \Granda ^2 \\ & \le \Const  \| \gamma _\varepsilon ^{\frac {1} {2}} u \| _{ L^\infty  }^\alpha .
\end{split} 
\end{equation} 
Applying~\eqref{fCLu} and the inequality $(x+y)^{\frac {\alpha } {2}}\le \Const( x^{\frac {\alpha } {2}} + y^{\frac {\alpha } {2}})$, we deduce from~\eqref{fLMmp} that 
\begin{equation} \label{fLMmuu} 
   \int _{ \R^N  }  \gamma _\varepsilon ^2 |u|^{\alpha +2} \le  
\Const \varepsilon ^{\frac {N\alpha } {2}}     
 + \Const \varepsilon ^{\frac {(N-1) \alpha } {2}}   
    \|  \gamma _\varepsilon  u_r\| _{ L^2 } ^{\frac {\alpha } {2}} .
\end{equation} 
We first consider the case $\alpha <4$. Applying the inequality $xy\le \frac { x^p} {p\delta ^p} + \frac {\delta ^{p'} y^{p'}} { p'}$ with $\delta >0$ and $p= \frac {4} {4-\alpha }$, $p'= \frac {4} {\alpha }$, we see that
\begin{equation*} 
\varepsilon ^{\frac {(N-1) \alpha } {2}}   
    \|  \gamma _\varepsilon  u_r\| _{ L^2 } ^{\frac {\alpha } {2}} \le \Const  \delta ^{- \frac {4} {4-\alpha }} \varepsilon ^{\frac {2(N-1)\alpha } {4-\alpha }} +\Const \delta ^{\frac {4} {\alpha }}    \|  \gamma _\varepsilon  u_r\| _{ L^2 } ^2
\end{equation*} 
so that~\eqref{fLMmuu} yields
\begin{equation} \label{fLMmuu1} 
   \int _{ \R^N  }  \gamma _\varepsilon ^2 |u|^{\alpha +2} \le \Const  \delta ^{\frac {4} {\alpha }}    \|  \gamma _\varepsilon  u_r\| _{ L^2 } ^2 +  \Const   \Bigl(  \varepsilon ^{\frac {N\alpha } {2}}   
 + \delta ^{- \frac {4} {4-\alpha }} \varepsilon ^{\frac {2(N-1)\alpha } {4-\alpha }}   \Bigr) .
\end{equation} 
Estimates~\eqref{fLMmud} and~\eqref{fLMmuu1} now yield
\begin{equation} \label{fLMmut} 
I_\varepsilon (u) \le -  \Bigl( 2-  \Const 
  \Avoir  \delta ^{\frac {4} {\alpha }}      \Bigr)  \|  \gamma _\varepsilon  u_r\| _{ L^2 } ^2   + \Const
 \Avoir    \Bigl(  \varepsilon ^{\frac {N\alpha } {2}}  
 +  [ \delta ^{-4} \varepsilon ^ {2(N-1)\alpha }  ] ^{\frac {1} {4-\alpha }} \Bigr) +\Const \varepsilon ^2  . 
\end{equation} 
We choose $\delta >0$ so that the first term in the right-hand side of~\eqref{fLMmut} vanishes, i.e.
$ \Const  \Avoir  \delta ^{\frac {4} {\alpha }}    =  2$. For this choice of $\delta $, 
 it  follows from~\eqref{fLMmut} that if $ \|u\| _{ L^2 }\le A$, then
$I_\varepsilon (u) \le  \kappa (\mu ,\varepsilon )$, where $\kappa $ is defined by~\eqref{eOgau:u1}.  
This proves inequality~\eqref{eOgau:u}  for $\alpha <4$. The case $\alpha =4$ follows by letting $\alpha \uparrow 4$ and observing that $ [ \Const    \Avoir  \varepsilon ^ {2(N-1) }  ] ^{\frac {\alpha } {4-\alpha }} \to 0$ as $\alpha \uparrow 4$ when $\Const    \Avoir  \varepsilon ^ {2(N-1) }  <1$.
\end{proof}

\end{document}